\documentclass[a4paper]{amsart}

\usepackage[shortlabels]{enumitem}
\setlist[enumerate]{label={\arabic*.}}
\setlist[description]{font=\normalfont\slshape}
\usepackage{comment}
\usepackage[dvipsnames]{xcolor} 
\usepackage[backref]{hyperref}
\hypersetup{colorlinks=true,linkcolor=Maroon,citecolor=OliveGreen,urlcolor=Maroon}

\usepackage{tikz}
\usetikzlibrary{matrix, positioning, fit, backgrounds, patterns, calc}

\usepackage{mathtools}
\usepackage{cleveref}

\usepackage{microtype}

\newtheorem{theorem}{Theorem}[section]
\newtheorem{lemma}[theorem]{Lemma}
\newtheorem{proposition}[theorem]{Proposition}

\theoremstyle{definition}
\newtheorem{question}[theorem]{Question}
\newtheorem{remark}[theorem]{Remark}

\newtheorem{example}[theorem]{Example}

\usepackage{amssymb}
\usepackage{amsthm}
\usepackage{mathtools}
\usepackage{amsmath,amsfonts,amsthm,mathtools,commath,amssymb}

\newcommand\opr[1]{\operatorname{#1}}

\def\C{\mathbf{C}}
\def\ZZ{\mathbf{Z}}
\def\F{\mathbf{F}}
\def\N{\mathbf{N}}
\def\PP{\mathbf{P}}

\def\GL{\opr{GL}}
\def\SL{\opr{SL}}

\def\diam{\opr{diam}}

\def\image{\opr{im}}
\def\ker{\opr{ker}}
\def\ad{\opr{ad}}

\def\tr{\opr{tr}}
\def\cspan{\opr{span}_{\C}}

\def\M{\opr{M}}
\def\diag{\opr{diag}}
\def\sl{\mathfrak{sl}}
\def\gl{\mathfrak{gl}}
\def\g{\mathfrak{g}}

\def\borel{\mathfrak{b}}
\def\D{\opr{D}}

\def\rank{\opr{rank}}

\def\zerodiag{\opr{M}^0}
\def\eval{\opr{eval}}
\def\End{\opr{End}}
\def\mon{\opr{mon}}
\def\ass{\opr{ass}}
\def\Sym{\opr{Sym}}

\def\Gr{\opr{Gr}}

\usepackage{todonotes}

\begin{document}
\baselineskip=13pt 

\title[Additive diameters]{Additive diameters of group representations}

\author{Urban Jezernik}
\address{Faculty of Mathematics and Physics, University of Ljubljana, Jadranska 19, 1000 Ljubljana, Slovenia 
/
Institute of Mathematics, Physics, and Mechanics, Jadranska 19, 1000 Ljubljana, Slovenia}
\email{urban.jezernik@fmf.uni-lj.si}

\author{Špela Špenko}
\address{Département de Mathématique, Université Libre de Bruxelles, Campus de la
Plaine CP 213, Bld du Triomphe, B-1050 Bruxelles, Belgium}
\email{spela.spenko@ulb.be}

\thanks{UJ was supported by the Slovenian Research Agency program P1-0222 and grants J1-50001, J1-4351, J1-3004, N1-0217. ŠŠ was supported by a MIS grant from the National Fund for Scientific Research
(FNRS) and an ARC grant from the Université Libre de Bruxelles.}

\begin{abstract}
We explore the concept of additive diameters in the context of group representations, unifying various noncommutative Waring-type problems. Given a finite-dimensional representation $\rho \colon G \to \GL(V)$ and a subspace $U \leq V$ that generates $V$ as a $G$-module, we define the \emph{$G$-additive diameter} of $V$ with respect to $U$ as the minimal number of translates of $U$ under the representation $\rho$ needed to cover $V$. We demonstrate that every irreducible representation of $\SL_2(\C)$ exhibits optimal additive diameters and establish sharp bounds for the conjugation representation of $\SL_n(\C)$ on its Lie algebra $\sl_n(\C)$. Additionally, we investigate analogous notions for additive diameters in Lie representations. We provide applications to additive diameters with respect to images of equivariant algebraic morphisms, linking them to the corresponding $G$-additive diameters of images of their differentials.
\end{abstract}

\maketitle

\section{Introduction}

\subsection{Additive diameters}

The study of additive decompositions in algebraic structures has a rich history. In an abelian semigroup $(S,+)$, one can ask if every element of $S$ is expressible as a sum of a bounded number of elements from a subset. For $X \subseteq S$, the \emph{additive diameter} of $S$ with respect to $X$ is the smallest $d$ such that every element of $S$ can be expressed as a sum of at most $d$ elements from $X$:
\[
  \diam_+(S,X) = \min \{ d \mid dX = S \}, \quad \text{where } dX = \{ x_1 + \cdots + x_d \mid x_i \in X \cup \{ 0 \} \}.
\]
This framework encompasses diverse problems across various structures. Here are some concrete examples.\footnote{Technical details behind some of the examples in the introduction are deferred to the appendix.}

\begin{example} \mbox{} \smallskip
\begin{itemize}[leftmargin=*]
  \item Let $X$ be the set of prime numbers. Goldbach's conjecture predicts that $2X \supseteq \{ 2n \mid n \in \N \}$. Helfgott's proof \cite{helfgott2013ternary} of the weak Goldbach conjecture shows $3X \supseteq \{ 2n + 1 \mid n \in \N \}$.
   
  \item Waring's problem bounds the additive diameters of power subsets of $\N_0$. Let $X = \{ x^k \mid x \in \N_0 \}$ for $k \geq 2$. Lagrange proved in 1770 that  $\diam_+(\N_0,X)$ is $4$ when $k = 2$, and Hilbert \cite{hilbert1909beweis} proved it is bounded in terms of $k$ for every $k$. 
  
  \item Let $X = \{ x \in \sl_n(\C) \mid x^2 = 0 \}$ be the set of square-zero traceless matrices with $n \geq 5$. Then $\diam_+(\sl_n(\C), X) = 4$ \cite{wang1991sums}.
  
  \item Let $f \colon \C^2 \to \C^4$ be the algebraic morphism $f(x,y) = (x^3, x^2 y, x y^2, y^3)$. Its image is a subvariety of $\C^4$ whose projectivization in $\PP^3$ is the twisted cubic. We have $\diam_+(\C^4, \image(f)) = 3$ (\Cref{section: examples}).
\end{itemize}
\end{example}

These examples illustrate how specific subsets (e.g., powers, primes, nilpotents, subvarieties) can quickly generate large structures under addition.

  \subsection{Diameters in finite groups}

  Beyond abelian semigroups, similar diameter questions arise in nonabelian group theory. For a finite group $G$ with a generating set $X$ containing the identity, the \emph{diameter} is the smallest $d$ such that every element of $G$ can be expressed as a word of length at most $d$ in the generators:
  \[
    \diam(G,X)  = \min \{ d \mid X^d = G \}, \quad \text{where } X^d = \{ x_1 x_2 \cdots x_d \mid x_i \in X \cup \{ 1 \}  \}.
  \]
  Babai's conjecture \cite{Babai1992} posits that every nonabelian finite simple group $G$ has a very small diameter (polylogarithmic in $|G|$) with respect to \emph{any} generating set. While still open, significant progress has been made, particularly for groups of Lie type of bounded rank (e.g., $\SL_n(\F_p)$ with fixed $n$ and $p \to \infty$). The driving force behind this are theorems on growth of sets under multiplication.

\begin{theorem}[Product theorem, \cite{helfgott2008growth, PyberSzabo, BreuillardGreenTao}]
Let $G$ be a finite simple group of Lie type of bounded rank. There exists $\epsilon > 0$, depending only on the rank of $G$, such that for any generating set $X \subseteq G$, we have
\[
|X^3| > |X|^{1 + \epsilon} \qquad \text{or} \qquad
|X| > |G|^{1 - \epsilon} \ \text{(in which case $X^3 = G$)}.
\]
\end{theorem}
The first case of the product theorem asserts that $X$ ``uniformly expands'' under multiplication. In the second case, $X$ is already so large that it covers $G$ in $3$ steps. Behind this is Gowers' argument \cite{Gowers} exploiting the high dimensions of the lowest-degree representations of $G$, leading to the strong result $X^3 = G$ \cite{NikolovPyber}. These results prove Babai's conjecture for bounded Lie rank and even support the construction of expander families \cite{bourgain2008uniform, BreuillardGreenGuralnickTao}.

Stronger results hold for generating subsets $X$ of finite simple groups consisting of specific elements, such as commutators, conjugacy classes, or images of word maps. For $X = \image(w)$, where $w$ is a word in the free group evaluated in group elements, Larsen, Shalev, and Tiep \cite{LarsenShalevTiep2011} show that for any nontrivial $w$, the diameters of all nonabelian finite simple groups of sufficiently large order are at most $2$. This result is a noncommutative analogue of the classical Waring problem.

\subsection{Diameters in infinite groups}

In finitely generated infinite groups, the analogue of ``uniform expansion'' is uniform exponential growth, meaning $\abs{X^d} \geq \omega^d$ for some $\omega > 1$, for all $d$, and for \emph{all} finite generating sets $X$ containing the identity. Eskin, Mozes, and Oh \cite{EskinMozesOh} showed that any finitely generated subgroup of $\GL_n(\C)$ either is virtually solvable or exhibits uniform exponential growth. For example, $\SL_n(\ZZ)$ has uniform exponential growth. The product theorem borrows ideas from these results.

Gowers' argument extends to compact groups: subsets with sufficiently large Haar measure cover the group under multiplication in a bounded number of steps (see \cite[Theorem 1.7]{ellis2024product} for a recent result).

Stronger results hold for generating subsets $X$ consisting of specific elements. For complex linear groups, Borel's theorem asserts that the image of a nontrivial word map is always dominant for connected semisimple groups. This dominance implies, via Borel's trick, that the diameter of the image of a nontrivial word map is at most $2$. This result is mirrored in the Larsen-Shalev-Tiep theorem for finite simple groups.

\subsection{Additive diameters in algebras}

In rings and algebras, additive decompositions have been widely studied, particularly in matrix algebras. A key focus is the noncommutative Waring problem, which examines the additive diameter of an algebra with respect to $X = \image(f)$, where $f$ is a noncommutative polynomial.

For finite matrix algebras, the problem has been explored for power words. For instance, \cite{KishoreSingh} shows that for $k \geq 1$, every matrix in $\M_n(\F_q)$ can be expressed as a sum of two $k$-th powers, provided $q$ is sufficiently large.

Over $\C$, the problem is better understood. Brešar, Šemrl, and Volčič \cite{brevsar2023waring,brevsar2023waringISRAEL,brevsar2024matrix} study $\M_n(\C)$ for arbitrary noncommutative polynomials $f$. A key result asserts that every matrix in $\sl_n(\C)$ can be written as the difference of two elements from $\image(f)$ if $n$ is large relative to $\deg(f)$. On the multiplicative side, it is shown that every nonscalar matrix in $\GL_n(\C)$ can be expressed as a product of two elements from $\image(f)$ under similar conditions.

For Lie algebras, related results appear in \cite{Bandman-Gordeev-Kunyavskii-Plotkin} for Lie polynomials. If $P$ is a nontrivial Lie polynomial that is not a polynomial identity in $\sl_2(\C)$, the induced polynomial map is dominant on any Chevalley Lie algebra (e.g., $\sl_n(\C)$ for $n \geq 2$). This is an infinitesimal analogue of Borel's theorem.

\subsection{Contributions -- Additive diameters in representations}

In this paper, we extend and unify some of the previous themes by studying additive diameters in the context of group representations. 

\subsubsection{Definition and examples}

Let $G$ be a group with a finite-dimensional representation $\rho \colon G \to \GL(V)$, and let $U$ be a subspace of $V$ that generates $V$ as a $G$-module. The \emph{$G$-additive diameter} of $V$ with respect to $U$ is the smallest number of translates of $U$ under $\rho$ needed to cover $V$:
\[
  \diam_+^G(V,U) = \min \left\{ d \;\middle|\; \rho(g_1) \cdot U + \dots + \rho(g_d) \cdot U = V \text{ for some $g_1,\dots,g_d \in G$}\right\}.
\]
The diameter is \emph{optimal} if $\diam_+^G(V,U) = \lceil \dim V / \dim U \rceil$. If this holds for \emph{every} subspace $U \leq V$, we say $V$ exhibits \emph{optimal $G$-additive diameters}.

Below, we provide examples of representations with both optimal and nonoptimal diameters, illustrating the challenges of identifying the necessary translates of $U$. These examples show that not all irreducible representations exhibit optimal diameters, even for large subspaces.

\begin{example} \label{diam of sl wrt zero diagonal} \mbox{} \smallskip
\begin{itemize}[leftmargin=*]
  \item Let $G = \GL_n(\C)$ act on $V = \C^n$ by matrix multiplication, and let $U \le V$ be any $d$-dimensional subspace with $d < n$. Since $\GL_n(\C)$ acts transitively on such subspaces, $V$ exhibits optimal $G$-additive diameters.
  
  \item Let $\GL_2(\C)$ act by conjugation on $\M_2(\C)$, and let $U = \cspan \langle I, E_{12} \rangle$. Since every conjugate of $U$ contains $I$, the diameter is at least $3$, which is not optimal. On the other hand, three conjugates of $U$ can cover $\M_2(\C)$ (\Cref{section: examples}), so $\diam_+^{\GL_2(\C)}(\M_2(\C), U) = 3$. Note that this representation is not irreducible.

  \item Let $\GL_n(\C)$ act by conjugation on $\sl_n(\C)$. For $U = \zerodiag_n(\C)$ (matrices with zero diagonal), we get the optimal diameter $\diam^{\GL_n(\C)}_+(\sl_n(\C), U) = 2$ (\Cref{section: examples}).

  \item Let $\GL_n(\C)$ act by conjugation on $\sl_n(\C)$, and let $U$ be the subspace of matrices with zero last row and column.\footnote{We thank Peter Šemrl for showing us this example.} This space is of dimension $(n-1)^2 - 1$, yet the diameter is $\diam_+^{\GL_n(\C)}(\sl_n(\C), U) = 3$, which is not optimal (\Cref{section: examples}).
\end{itemize}
\end{example}

We demonstrate that for $\SL_2(\C)$, non-optimal behavior never occurs.

\begin{theorem}
Every irreducible representation of $\SL_2(\C)$ exhibits optimal group-additive diameters.
\end{theorem}

\subsubsection{Dimension expanders}

Bounding these diameters can be approached in two stages. The first step involves establishing ``uniform expansion'' of the subspace $\sum_{1 \leq i \leq k} \rho(g_i) \cdot U$ as $k$ increases, measured by its dimension. This dimension expansion, analogous to the first case of the product theorem in groups, has been studied by Lubotzky and Zelmanov \cite{LubotzkyZelmanov} and relies on expander families and their unitary representations.

\begin{theorem}[Dimension expanders, Proposition 2.1 in \cite{LubotzkyZelmanov}]
Let $G$ be a group generated by a finite set $S$ with Kazhdan constant\footnote{The Kazhdan constant is $\inf_{\rho} \inf_{0 \neq v \in V} \max_{s \in S} \norm{\rho(s) \cdot v - v}/\norm{v}$, where the infimum runs over unitary representations $\rho$ of $G$ on a Hilbert space $V$ without non-trivial $G$-fixed points.} $\epsilon > 0$. For any irreducible unitary representation $\rho \colon G \to \GL(V)$, we have
\[
  \textstyle \dim \left( U + \sum_{s \in S} \left( \rho(s) \cdot U \right) \right) \geq (1 + \epsilon^2/12) \cdot \dim U
\]
for all subspaces $U \leq V$ of dimension at most $\dim V / 2$. 
\end{theorem}

Kassabov \cite{kassabov2007symmetric} showed that symmetric groups $S_n$ are expanders with respect to explicitly constructed generating sets of size $k \leq 30$, implying uniformly bounded Kazhdan constants. Thus, the dimension expanders theorem applies to $G = S_n$ with its conjugation action on the irreducible representation $\C^{n-1}$ induced by the natural permutation representation on $\C^n$. Recently, Reineke \cite{Reineke1} further developed the above theorem and optimized the expansion constant using quiver representations.

\subsubsection{Completing the covering}

The results above demonstrate that starting with a subspace $U \leq V$, we can cover at least half of $V$ using translates of $U$ under $G$. However, beyond this point, the theorem no longer guarantees uniform dimension expansion. Indeed, for dimension reasons, uniform expansion cannot hold for all subspaces of sufficiently large dimension. To fully cover $V$ in a bounded number of steps, we require a variant of the second part of the product theorem or Gowers' argument. This is the primary goal of the present paper.

We focus on addressing the conjugation representation of $\SL_n(\C)$ on its Lie algebra $\sl_n(\C)$. We show that sufficiently large subspaces exhibit optimal diameters.

\begin{theorem}
Let $U \in \Gr(\sl_n(\C), d)$ with $(n-1)^2 < d < n^2 - 1$, $n\neq 4$.
Then 
  \[
  \diam_+^{\SL_n(\C)}(\sl_n(\C), U) = 2.
  \]
\end{theorem}

Surprisingly, the dimension bound is sharp. We provide examples of subspaces of dimension $(n-1)^2$ with diameter $3$. We then generalize the above result to all large subspaces, exhibiting optimal group-additive diameters up to logarithm factor.

\begin{proposition}
Let $0<\epsilon<1/3$ and $n\geq \max\{9,2/\epsilon\}$. If
$U\leq \sl_n(\C)$ has $\dim U>\epsilon n^2$, then
\[
  \diam_+^{\SL_n(\C)}
  \bigl(\sl_n(\C),U\bigr)
  \leq \frac{3456}{\epsilon}
  \left\lceil \log_2\frac{4}{\epsilon}\right\rceil .
\]
\end{proposition}

\subsubsection{Lie-additive diameters}

We also explore additive diameters in Lie algebras and their representations, drawing parallels and highlighting contrasts with the group-theoretic results. These problems are often more tractable. We provide several nonequivalent definitions of Lie-additive diameters. For the most natural definition, we show that irreducible representations of $\sl_2(\C)$ exhibit optimal diameters, similar to the group case, but with a significantly simpler argument. However, the precise relationship between group-additive and Lie-additive diameters remains unclear, raising further questions about their connection.

\subsection{Applications -- Additive diameters in equivariant morphisms}

Finally, we explore applications of the results above to additive diameters of vector spaces with respect to images of equivariant morphisms. This is a representation-theoretic analogue of the Waring problem and generalizes the results in the noncommutative setting discussed earlier.

Suppose $G$ acts on vector spaces $W$ and $V$ via representations. A morphism between these representations is a $G$-equivariant linear map $f \colon W \to V$, meaning $f(g \cdot w) = g \cdot f(w)$ for all $g \in G$ and $w \in W$. Here, we consider an extension of these to polynomial maps (algebraic morphisms) and show that the additive diameter of $V$ with respect to $\image(f)$ can be bounded in terms of the $G$-additive diameter of $V$ with respect to the image of the differential of $f$.

\begin{theorem}
  Let $G$ be a complex linear algebraic group. Let $f \colon W \to V$ be a $G$-equivariant algebraic morphism. Then, for any $w \in W$,
  \[
    \diam_+(V, \image(f)) \leq 2 \cdot \diam^G_+(V, \image (D_w f)).
  \]
\end{theorem}

We further show how the images $\image (D_w f)$ can be understood in terms of the derivative of the representation $\rho$.

This theorem unifies several known results. For example, applied to the conjugation representation of $\SL_n(\C)$ on $\M_n(\C)$, it recovers (up to constants) some results from \cite{brevsar2023waring,brevsar2023waringISRAEL,brevsar2024matrix} on the noncommutative Waring problem in matrix algebras, as noncommutative polynomials are specific examples of equivariant morphisms. Additionally, when $Z$ is a homogeneous variety in $V$ with a $G$-equivariant polynomial parametrization $f \colon W \to Z \subseteq V$, this theorem can be compared to Terracini's lemma \cite[Proposition 10.10]{eisenbud20163264}, linking tangent spaces of secant varieties to images of differentials of parametrizations.

\subsection{Reader's Guide}

We begin by outlining our strategy for bounding group-additive diameters (\Cref{section: borel invariant subspaces}). The method used to prove that all subspaces of a given dimension satisfy a diameter bound is the Borel fixed point theorem, applied to the variety of potential counterexamples in the Grassmannian of subspaces in $\sl_n(\C)$. This approach, originating from \cite{draisma2006nilpotent}, shows that if counterexamples exist, they must take a specific form. By analyzing these forms, we demonstrate they are in fact not counterexamples. Next, we establish the main results on group-additive diameters, starting with the irreducible representations of $\SL_2(\C)$ (\Cref{section: irreducible representations of SL2}). Intriguingly, the proof here relies on the well-posedness of the Hermite interpolation problem, a classical result from numerical analysis. We then inspect conjugation representations of $\SL_n(\C)$ with respect to large subspaces (\Cref{section: conjugation large subspaces} for optimal diameters and \Cref{section: conjugation largish subspaces} for the general case). Subsequently, we turn to analogous results for Lie-additive diameters (\Cref{section: Lie-additive diameters}). Finally, we explore applications of these results to equivariant morphisms (\Cref{section: equivariant morphisms}). Technical details behind the examples in the introduction are contained in the appendix (\Cref{section: examples}).

\subsection{Acknowledgements}
This project benefited from discussions with several people. We thank Matej Brešar, Sean Cotner, Marjetka Knez, Tomaž Košir, Primož Potočnik, Peter Šemrl, Jurij Volčič, and especially Klemen Šivic for their valuable ideas, relevant references, and their interest in the questions studied in this note. We thank Ernesto Ingrosso for many helpful comments on the first version of this paper, and in particular for pointing out a mistake in the proof of Proposition \ref{epsilon n^2 proposition}, which we have now replaced by a weaker statement. We also acknowledge the use of artificial intelligence (ChatGPT, version o3-mini-high) for identifying connections with existing concepts in the literature and improving the exposition.

\section{Borel Stable Subspaces} \label{section: borel invariant subspaces}

\subsection{Strategy for bounding diameters}

Let $G$ be a complex linear algebraic group with a representation $\rho \colon G \to \GL(V)$. Suppose we wish to prove that the $G$-additive diameter of $V$ with respect to \emph{all} subspaces $U \leq V$ of a given dimension $d$ is at most $k$. For the sake of contradiction, suppose there is a counterexample. Collect all the potential counterexamples into a set
\[
  X_{d,k} = \{ U \in \Gr(V, d) \mid \diam_+^G(V, U) > k \},
\]
where $\Gr(V,d)$ is the Grassmanian variety of $d$-dimensional subspaces of $V$.

\begin{lemma}
The set $X_{d,k}$ is a closed subvariety of $\Gr(V, d)$.
\end{lemma}
\begin{proof}
Let $\mathcal B$ be an ordered basis of $V$, and let $J$ be a $d$-tuple of distinct indices from $\{ 1, \dots, \dim V \}$. Let $M_J$ be the set of $d \times \dim V$ matrices whose columns indexed by $J$ form the identity matrix. The row spans of matrices in $M_J$ form an open affine subset $\mathcal O_J \subseteq \Gr(V, d)$. These open subsets give a chart for $\Gr(V, d)$ as $J$ varies. Let
\[
  Y_{d,k} = \{ (U, g_1, \dots, g_k) \in \Gr(V, d) \times G^k \mid \textstyle \sum_{1 \leq i \leq k} \rho(g_i) \cdot U = V \}.
\]
The set $Y_{d,k} \cap (\mathcal O_J \times G^k)$ corresponds to matrices $M_J$ with rows $m_1, \dots, m_d$ (where $m_i \in V$) for which the matrix with rows $\rho(g_i) \cdot m_j$ for $1 \leq i \leq k$ and $1 \leq j \leq d$ has full rank. This is an open condition, so $Y_{d,k}$ is an open subset of $\Gr(V, d) \times G^k$. Consider the projection $\Gr(V,d) \times G^k \to \Gr(V, d)$ onto the first component. This is an open map, so its image, consisting of subspaces $U \in \Gr(V, d)$ that cover $V$ with at most $k$ conjugates, is open in $\Gr(V, d)$. Its complement $X_{d,k}$ is then closed.
\end{proof}

Note that the variety $X_{d,k}$ is closed under the action of $G$ on $\Gr(V,d)$. Supposing this set is not empty, it must then contain a very particular subspace.\footnote{We thank Klemen Šivic for highlighting this connection and further developing it in \cite{omladivc2023approximate} for $\GL_n(\C)$ acting by conjugation on $\M_n(\C)$.}

\begin{theorem}[Borel fixed point theorem, III.10.4 in \cite{borel2012linear}]
Let $X \subseteq \Gr(V, d)$ be a nonempty closed subvariety that is invariant under the action of $G$ on $\Gr(V,d)$. Then $X$ contains a subspace $U$ stable under the action of the Borel subgroup of $G$.\footnote{This means that for all $g$ in the Borel subgroup, we have $\rho(g) \cdot u \in U$ for all $u \in U$.}
\end{theorem}
Therefore, to show that $X_{d,k} = \emptyset$, it is sufficient to prove that every subspace $U \in \Gr(V, d)$ that is stable under the Borel subgroup of $G$ satisfies $\diam_+^G(V, U) \leq k$. This approach is considerably simpler than proving the same for all subspaces in $\Gr(V, d)$, because in many cases we can analyze the structure of these stable subspaces and directly verify that each of them has diameter at most $k$.

\begin{proposition} \label{diameter bound reduction to Borel fixed subspaces}
Let $G$ be a complex linear algebraic group with a representation on $V$. Suppose that every subspace $U \in \Gr(V, d)$ that is stable under the action of the Borel subgroup of $G$ satisfies $\diam_+^G(V, U) \leq k$. Then \emph{every} subspace in $\Gr(V,d)$ satisfies the same conclusion.
\end{proposition}

\subsection{Examples}

\subsubsection{Irreducible representations of $\SL_2(\C)$}

Irreducible smooth representations of $\SL_2(\C)$ are given by heighest weight theory. For any $k \geq 1$, there is a unique irreducible representation of dimension $k+1$. It can be seen as the $k$-th symmetric power of the standard representation of $\SL_2(\C)$ on $\C^2$, and it can be realized as a representation on the space of homogeneous polynomials $\C[X,Y]_k$ of degree $k$ in two variables. A basis for this spaces is $e_i = X^i Y^{k-i}$ for $0 \leq i \leq k$. We have a representation $\rho_k$ of $\SL_2(\C)$ on $\C[X,Y]_k$ given by
\[
  \rho_k \begin{pmatrix}
    a & b \\
    c & d
  \end{pmatrix}
  \cdot X^i Y^{k-i}
  = (aX + cY)^i (bX + dY)^{k-i}.
\]

Let $U$ be a subspace of $\C[X,Y]_k$ that is stable under the Borel subgroup of $\SL_2(\C)$. In particular, the standard maximal torus fixes $U$. Hence, if $u = \sum_{0 \leq i \leq k} \alpha_i e_i$ belongs to $U$, then so does $\sum_{0 \leq i \leq k} \lambda^{2i - k} \alpha_i e_i$ for all $\lambda \neq 0$. By Vandemonde, we then have $\alpha_i e_i \in U$ for each $i$. Hence $U$ contains all $e_i$ for which $\alpha_i \neq 0$, and so $U$ is in fact spanned by some of the basis elements $e_i$. Let $j$ be the smallest index such that $e_j \in U$. We then have, by
\[
  \rho_k \begin{pmatrix}
    1 & 1 \\
    0 & 1
  \end{pmatrix}
  \cdot e_j = 
  \sum_{j \leq i \leq k} \binom{k-j}{i-j} e_i,
\]
that $U = \langle e_j, e_{j+1}, \dots, e_k \rangle$. 

Call a subspace $U \leq \C[X,Y]_k$ \emph{upper closed} if it is of the form $\langle e_j, e_{j+1}, \dots, e_k \rangle$ for some $j$ (see \Cref{figure: upper closed subspace for SL2}). Thus every subspace that is stable under the Borel subgroup of $\SL_2(\C)$ is upper closed. The converse is clearly true as well.

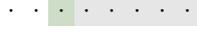
\begin{figure}[t]
  \begin{tikzpicture}
  
    \matrix[matrix of math nodes, nodes in empty cells] (m)
    {
      \cdot & \cdot & \cdot & \cdot & \cdot & \cdot & \cdot & \cdot \\
    };
  
    \foreach \c in {4,...,8} {
        \node[fill=gray!20] at (m-1-\c) {$\cdot$};
    }
  
    \node[fill=OliveGreen!20] at (m-1-3) {$\cdot$};
  \end{tikzpicture}
  \caption{The upper closed subspace $\langle e_3, e_4, \dots, e_8 \rangle \leq \C[X,Y]_8$.}
  \label{figure: upper closed subspace for SL2}
\end{figure}

\subsubsection{Conjugation of $\SL_n(\C)$ on $\sl_n(\C)$}

Take $G = \SL_n(\C)$ and let is act by conjugation on its Lie algebra $V = \sl_n(\C)$. This is an irreducible representation. Subspaces $U \leq V$ that are stable under the Borel subgroup have been described in detail in \cite[Lemma 9]{omladivc2023approximate}. Let us recall here how these look like.

For any $1 \leq i, j \leq n$, let 
\[
    B_{ij} = \cspan \langle E_{k\ell} \mid k \leq i, \ \ell \geq j \rangle \cap \sl_n(\C)
\]
be the subspace of traceless matrices whose nonzero entries are in the upper right block with corner $(i, j)$. Call such a subspace an \emph{upper right block subspace} (see \Cref{figure: upper_right_block_subspace}). These spaces can alternatively be described in terms of the standard Borel subalgebra $\borel$ of $\sl_n(\C)$ consisting of upper triangular matrices. We have
\[
[\borel, E_{ij}] = 
\cspan \langle E_{k j} \mid k < i, \ k \neq j \rangle
+ \cspan \langle E_{i k} \mid k > j, \ k \neq i \rangle
+ \cspan \langle E_{ii} - E_{jj} \rangle,
\]
and so $B_{ij}$ is precisely the submodule of $\sl_n(\C)$ generated by $E_{ij}$ under the action of $\borel$ on $\sl_n(\C)$ by adjoint representation.

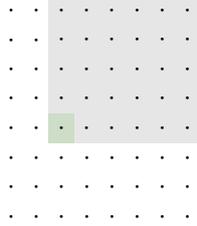
\begin{figure}[t]
    \begin{tikzpicture}
      \def\n{8}  
      \def\i{5}  
      \def\j{3}  
    
      \matrix[matrix of math nodes, nodes in empty cells] (m)
      {
        \cdot & \cdot & \cdot & \cdot & \cdot & \cdot & \cdot & \cdot \\
        \cdot & \cdot & \cdot & \cdot & \cdot & \cdot & \cdot & \cdot \\
        \cdot & \cdot & \cdot & \cdot & \cdot & \cdot & \cdot & \cdot \\
        \cdot & \cdot & \cdot & \cdot & \cdot & \cdot & \cdot & \cdot \\
        \cdot & \cdot & \cdot & \cdot & \cdot & \cdot & \cdot & \cdot \\
        \cdot & \cdot & \cdot & \cdot & \cdot & \cdot & \cdot & \cdot \\
        \cdot & \cdot & \cdot & \cdot & \cdot & \cdot & \cdot & \cdot \\
        \cdot & \cdot & \cdot & \cdot & \cdot & \cdot & \cdot & \cdot \\
      };
    
      \foreach \r in {1,...,\i} {
        \foreach \c in {\j,...,\n} {
          \node[fill=gray!20] at (m-\r-\c) {$\cdot$};
        }
      }
    
      \node[fill=OliveGreen!20] at (m-\i-\j) {$\cdot$};
    \end{tikzpicture}
    \caption{The upper right block subspace $B_{53}$.}
    \label{figure: upper_right_block_subspace}
\end{figure}

Say a subspace $U \in \Gr(\sl_n(\C), d)$ is \emph{upper right block closed} if whenever $u \in U$ and $u_{ij} \neq 0$ with $i \neq j$, then $B_{ij} \leq U$. Every upper right block closed subspace is a sum of upper right block subspaces and a diagonal subspace. It is proved in \cite[Lemma 9]{omladivc2023approximate} that every subspace that is stable under the Borel subgroup is upper right block closed. The converse is clearly true as well.

\section{Irreducible Representations of $\SL_2(\C)$}
\label{section: irreducible representations of SL2}

Here, we prove that irreducible representations of $\SL_2(\C)$ always exhibit optimal diameters with respect to any subspace.

\begin{theorem}\label{thm:optimal_SL_2}
Every irreducible representation of $\SL_2(\C)$ exhibits optimal group-additive diameters.
\end{theorem}
\begin{proof}
Let $V = \C[X,Y]_k$ for some $k \geq 1$. For simplicity, let us further identify $\C[X,Y]_k$ with polynomials of degree at most $k$, denoted as $\C[x]_{\leq k}$, under $X = x, Y = 1$. By \Cref{diameter bound reduction to Borel fixed subspaces}, it suffices to prove the claim for every upper closed subspace $U = \langle e_j, e_{j+1}, \dots, e_k \rangle$, which is, by the above identification, equal to $\langle x^j,x^{j+1},\dots,x^k\rangle$. For any $0 \leq i < d$ and $a \in \C$, we have
  \[
    \rho_k \begin{pmatrix}
      1 & 0 \\
      a & 1
    \end{pmatrix}
    \cdot x^{k - i} =(x + a)^{k - i}.
  \]
Let $d = \dim U = k - j + 1$. We can assume $d < k + 1$. Let $n= \lceil (k+1)/d \rceil$. We claim that for distinct $a_1, \dots, a_n \in \C$, we have
\[
\sum_{0 \leq \ell < n}
\left(
\rho_k\begin{pmatrix}
  1 & 0 \\
  a_\ell & 1
\end{pmatrix} \cdot U 
\right)
= V.
\]
This is equivalent to saying that the polynomials $(x+a_\ell)^{k-i}$ for $0 \leq \ell < n$ and $0 \leq i < d$ span $\C[x]_{\leq k}$, which is in turn the same as
\[
\cspan \left\langle \partial_x^i ( (x+a_\ell)^k ) \mid 0\leq \ell < n, \ 0\leq i < d \right\rangle
= \C[x]_{\leq k},
\]  
where $\partial_x$ is the derivative operator. The latter holds by the following lemma.
\end{proof}

\begin{lemma}\label{lem:poly_basis}
Let $p \in \C[x]$ be of degree $k$. Let $d<k+1$ and $n=\lceil\frac{k+1}{d}\rceil$. Then
\[
\cspan \left\langle \partial_x^i ( p(x+a_\ell) ) \mid 0\leq \ell < n, \ 0\leq i < d \right\rangle
= \C[x]_{\leq k}
\] 
for any distinct $a_1, \dots, a_n \in \C$.
\end{lemma}

\begin{proof}
Using the Taylor series, we can expand the polynomial $p(x+a)$ as
\[
p(x+a)=\sum_{0 \leq j \leq k} \frac{1}{j!} a^j \partial_x^j ( p(x) ).
\]
Hence
\[
\partial_x^i ( p(x+a) )   
= \partial_a^i ( p(x+a) ) 
= \sum_{0 \leq j \leq k} \partial_a^i ( a^j ) \cdot \frac{1}{j!} \partial_x^j ( p(x) ).
\]
We can therefore express $\partial_x^i( p(x+a) )$ in the basis $\partial_x^j( p(x) )/j!$ of $\C[x]_{\leq k}$ with the matrix of coefficients
\[
  V_a = \left[ \partial_a^i (a^j) \right]_{0 \leq j < k, 0 \leq i < d} \in \M_{(k+1) \times d}(\C).
\]
Now, for $a_1, \dots, a_n \in \C$, the block matrix
\[
V = \left[V_{a_1}, \dots, V_{a_n}\right] \in \M_{(k+1) \times nd}(\C)
\]
is exactly the confluent Vandermonde matrix corresponding to the Hermite interpolation problem (see \cite{kalman1984generalized}). Taking $a_\ell$ to be distinct, the first $k+1$ columns of $V$ form an invertible matrix.
\end{proof}

\section{Conjugation -- Large Subspaces with Optimal Diameters}
\label{section: conjugation large subspaces}

In this section, we study the conjugation representation of $\GL_n(\C)$ on $\sl_n(\C)$. We prove that if a proper subspace $U \lneqq \sl_n(\C)$ has sufficiently large dimension (specifically, $\dim U > (n-1)^2$), then its $\GL_n(\C)$-additive diameter is $2$, which is optimal. We also demonstrate that this dimension threshold is sharp by providing an example of a subspace of dimension $(n-1)^2$ whose additive diameter is $3$.

\subsection{Diameter $2$ at dimension larger than $(n-1)^2$}

\begin{lemma} \label{upper right block from diagonal entry}
  Let $U \in \Gr(\M_n(\C), d)$ be upper right block closed. Let $\delta = d - 3n^2/4 - n/2 + 1/4 > 0$. For every $k$ with $\abs{k - (n+1)/2} < \sqrt{\delta}$, we have $B_{ij} \leq U$ for some $i \neq j$ with $i > k > j$.
  \end{lemma}
  \begin{proof}
  There are $(k-1)(n-k)$ elementary matrices $E_{ij}$ with $i \neq j$ and $i > k > j$. The condition $\abs{k - (n+1)/2} < \sqrt{\delta}$ is equivalent to $(k-1)(n-k) > n^2 - d$. Hence the space $U$ contains a matrix $u$ with $u_{ij} \neq 0$ for some $i \neq j$ with $i > k > j$. As $U$ is upper right block closed, we thus have $B_{ij} \leq U$.
  \end{proof}  

Using the lemma, we show that every large upper right block closed subspace be of a particular form. 
These forms are related to the standard Borel subalgebra $\borel$ of $\sl_n(\C)$ consisting of upper triangular matrices. Let $\borel_{\widehat{11}}$ denote the elements of $\borel$ with zero $(1,1)$ entry, and let $\borel_{\widehat{nn}}$ be elements of $\borel$ with zero $(n,n)$ entry.

\begin{proposition} \label{critical upper right closed subspaces}
Let $U \in \Gr(\sl_n(\C), d)$ be upper right block closed with $d = (n-1)^2 + 1$. Then one of the following holds (see \Cref{figure: critical_upper_right_closed_subspaces}):
\[
  \borel_{\widehat{11}} \leq U; \qquad
  \borel_{\widehat{nn}} \leq U; \qquad
  U = B_{21} + B_{n3}; \qquad
  U = B_{n,n-1} + B_{n-2,1}.
\]
\end{proposition}

\begin{figure}[t]
  \begin{tikzpicture}
    \def\n{8}
    
    \matrix[matrix of math nodes, nodes in empty cells] (m)
    {
      \cdot & \cdot & \cdot & \cdot & \cdot & \cdot & \cdot & \cdot \\
      \cdot & \cdot & \cdot & \cdot & \cdot & \cdot & \cdot & \cdot \\
      \cdot & \cdot & \cdot & \cdot & \cdot & \cdot & \cdot & \cdot \\
      \cdot & \cdot & \cdot & \cdot & \cdot & \cdot & \cdot & \cdot \\
      \cdot & \cdot & \cdot & \cdot & \cdot & \cdot & \cdot & \cdot \\
      \cdot & \cdot & \cdot & \cdot & \cdot & \cdot & \cdot & \cdot \\
      \cdot & \cdot & \cdot & \cdot & \cdot & \cdot & \cdot & \cdot \\
      \cdot & \cdot & \cdot & \cdot & \cdot & \cdot & \cdot & \cdot \\
    };
  
    \foreach \i in {1,...,\n} {
      \pgfmathtruncatemacro{\start}{ifthenelse(\i<2,2,\i)}
      \foreach \j in {\start,...,\n} {
        \node[fill=gray!20] at (m-\i-\j) {$\cdot$};
      }
    }
    
  
  \end{tikzpicture}
  \
  \begin{tikzpicture}
      \def\n{8}
      
      \matrix[matrix of math nodes, nodes in empty cells,] (m)
      {
        \cdot & \cdot & \cdot & \cdot & \cdot & \cdot & \cdot & \cdot \\
        \cdot & \cdot & \cdot & \cdot & \cdot & \cdot & \cdot & \cdot \\
        \cdot & \cdot & \cdot & \cdot & \cdot & \cdot & \cdot & \cdot \\
        \cdot & \cdot & \cdot & \cdot & \cdot & \cdot & \cdot & \cdot \\
        \cdot & \cdot & \cdot & \cdot & \cdot & \cdot & \cdot & \cdot \\
        \cdot & \cdot & \cdot & \cdot & \cdot & \cdot & \cdot & \cdot \\
        \cdot & \cdot & \cdot & \cdot & \cdot & \cdot & \cdot & \cdot \\
        \cdot & \cdot & \cdot & \cdot & \cdot & \cdot & \cdot & \cdot \\
      };
    
      \foreach \i in {1,...,7} {
        \foreach \j in {\i,...,\n} {
          \node[fill=gray!20] at (m-\i-\j) {$\cdot$};
        }
      }
      
    \end{tikzpicture}
    \ 
\begin{tikzpicture}
  \def\n{8}  
  
  \matrix[matrix of math nodes, nodes in empty cells] (m)
  {
    \cdot & \cdot & \cdot & \cdot & \cdot & \cdot & \cdot & \cdot \\
    \cdot & \cdot & \cdot & \cdot & \cdot & \cdot & \cdot & \cdot \\
    \cdot & \cdot & \cdot & \cdot & \cdot & \cdot & \cdot & \cdot \\
    \cdot & \cdot & \cdot & \cdot & \cdot & \cdot & \cdot & \cdot \\
    \cdot & \cdot & \cdot & \cdot & \cdot & \cdot & \cdot & \cdot \\
    \cdot & \cdot & \cdot & \cdot & \cdot & \cdot & \cdot & \cdot \\
    \cdot & \cdot & \cdot & \cdot & \cdot & \cdot & \cdot & \cdot \\
    \cdot & \cdot & \cdot & \cdot & \cdot & \cdot & \cdot & \cdot \\
  };

  \foreach \r in {1,...,2} {
    \foreach \c in {1,...,\n} {
      \node[fill=gray!20] at (m-\r-\c) {$\cdot$};
    }
  }

  \foreach \r in {1,...,\n} {
    \foreach \c in {3,...,\n} {
      \node[fill=gray!20] at (m-\r-\c) {$\cdot$};
    }
  }

  \node[fill=OliveGreen!20] at (m-2-1) {$\cdot$};
  \node[fill=OliveGreen!20] at (m-\n-3) {$\cdot$};

  \node[pattern=north west lines, pattern color=Maroon] at (m-2-2) {$\cdot$};
  \node[pattern=north west lines, pattern color=Maroon] at (m-3-3) {$\cdot$};
\end{tikzpicture}
\
\begin{tikzpicture}
    \def\n{8}  
    
    \matrix[matrix of math nodes, nodes in empty cells,] (m)
    {
      \cdot & \cdot & \cdot & \cdot & \cdot & \cdot & \cdot & \cdot \\
      \cdot & \cdot & \cdot & \cdot & \cdot & \cdot & \cdot & \cdot \\
      \cdot & \cdot & \cdot & \cdot & \cdot & \cdot & \cdot & \cdot \\
      \cdot & \cdot & \cdot & \cdot & \cdot & \cdot & \cdot & \cdot \\
      \cdot & \cdot & \cdot & \cdot & \cdot & \cdot & \cdot & \cdot \\
      \cdot & \cdot & \cdot & \cdot & \cdot & \cdot & \cdot & \cdot \\
      \cdot & \cdot & \cdot & \cdot & \cdot & \cdot & \cdot & \cdot \\
      \cdot & \cdot & \cdot & \cdot & \cdot & \cdot & \cdot & \cdot \\
    };
  
    \foreach \r in {1,...,\n} {
      \foreach \c in {7,8} {
        \node[fill=gray!20] at (m-\r-\c) {$\cdot$};
      }
    }
  
    \foreach \r in {1,...,6} {
      \foreach \c in {1,...,\n} {
        \node[fill=gray!20] at (m-\r-\c) {$\cdot$};
      }
    }
  
    \node[fill=OliveGreen!20] at (m-\n-7) {$\cdot$};
    \node[fill=OliveGreen!20] at (m-6-1) {$\cdot$};

    \node[pattern=north west lines, pattern color=Maroon] at (m-6-6) {$\cdot$};
    \node[pattern=north west lines, pattern color=Maroon] at (m-7-7) {$\cdot$};
  \end{tikzpicture}
\caption{The spaces $\borel_{\widehat{11}}$, $\borel_{\widehat{nn}}$, $B_{21} + B_{n3}$ and $B_{n,n-1} + B_{n-2,1}$. Note that $E_{22} - E_{33} \notin B_{21} + B_{n3}$, as represented by the red hatching.}
\label{figure: critical_upper_right_closed_subspaces}
\end{figure}
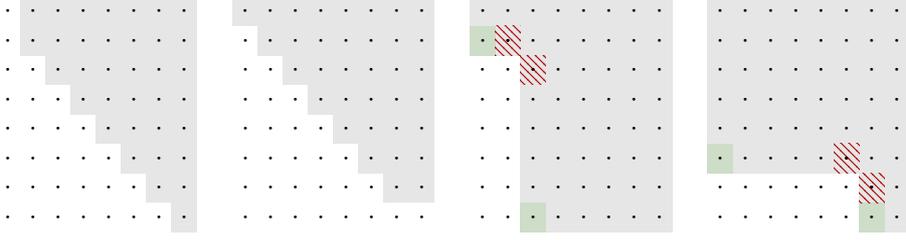

\begin{proof}
Let $\delta = d - 3n^2/4 - n/2 + 1/4 = n^2/4 - 5n/2 + 9/4$. If $n > 10$, we have $\sqrt{\delta} > (n+1)/2 - 4$ and so it follows from the previous lemma that for any $k$ with $4 \leq k \leq n - 3$, we have $B_{ij} \leq U$ for some $i \neq j$ with $i > k > j$. In particular, $U$ contains $E_{ii} - E_{jj}$ for all $3 \leq i < j \leq n - 2$. It can be verified by hand or with a computer that the same conclusion holds for all $6 \leq n \leq 10$. Let us now inspect the boundary indices of matrices in $U$.

\begin{description}
    \item[$E_{31}, E_{n,n-3} \in U$] In this case, $U$ contains the whole diagonal in $\sl_n(\C)$ and $\borel \leq U$.
    
    \item[$E_{31} \notin U$] In this case, every matrix $u \in U$ satisfies $u_{i1} = 0$ for $i \geq 3$. Let us inspect the second column.
    \begin{description}
        \item[$E_{32} \notin U$] Here, all matrices in $U$ have zeroes in entries $(i,j)$ for $i \geq 3$ and $j \leq 2$. There are a total of $2n-4$ of these entries. Since the codimension of $U$ in $\sl_n(\C)$ is $2n - 3$, we can either miss exactly one more $E_{ij}$ ($i\neq j$) or the dimension of the intersection of $U$ with the traceless diagonal is $n-2$. In the first case, since $U$ is upper right block closed, we cannot miss any $E_{ij}$ with $i<j$, thus $\borel \leq U$. In the second case, we have $E_{21}, E_{n,3} \in U$, and so $E_{ii}-E_{jj} \in U$ for all $3\leq i<j\leq n$ as well as $E_{11}-E_{22} \in U$. Thus, $U=B_{21}+B_{n3}$.

        \item[$E_{32} \in U$] Inspect the last row.
        \begin{description}
            \item[$E_{n,n-1} \notin U$] If $E_{21} \in U$, then $U$ contains $B_{21}$, and so $\borel_{\widehat{nn}} \leq U$. On the other hand, if $E_{21} \notin U$, then, for codimension reasons, $\borel \leq U$.
            
            \item[$E_{n,n-1} \in U$] In this case, we must have, for codimension reasons, that $E_{n-1,n-2} \in U$. Hence $E_{ii} - E_{jj} \in U$ for all $2 \leq i < j \leq n$, and so $\borel_{\widehat{11}} \leq U$.
        \end{description}
    \end{description}

    \item[$E_{n,n-3} \notin U$] This case is symmetric to the previous one after flipping all elements of $U$ along the antidiagonal of the matrix. \qedhere
\end{description}
\end{proof}

For each one of the cases above, we can easily bound the $\GL_n(\C)$-additive diameter of $\sl_n(\C)$ with respect to $U$ using a single matrix. 
Let $F = [\delta_{i + j = n + 1}]_{ij} \in \GL_n(\C)$ be the order $2$ permutation matrix corresponding to $\prod_{i < n/2} (i, \ n+1-i)$. For any $A \in \sl_n(\C)$, we have $(F A F)_{ij} = A_{n+1-i,n+1-j}$, and so conjugation by $F$ flips $A$ along the center of the matrix.

\begin{theorem} \label{diameter 2 at large dimension}
Let $U \in \Gr(\sl_n(\C), d)$ with $(n-1)^2 < d < n^2 - 1$.
Then 
\[
\diam_+^{\GL_n(\C)}(\sl_n(\C), U) = 2.
\]
\end{theorem}
\begin{proof}
We can assume $d = (n-1)^2 + 1$.
It is immediate that in all of the four options in the previous proposition, we have $U + F U F = \sl_n(\C)$. The result now follows from \Cref{diameter bound reduction to Borel fixed subspaces}.
\end{proof}

\subsection{Diameter $3$ at dimension $(n-1)^2$}

Following the particular form of critical upper right closed subspaces appearing in \Cref{critical upper right closed subspaces}, we now show that the bound in terms of the dimension is sharp. Namely, we exhibit a subspace $U$ of $\sl_n(\C)$ of dimension $(n-1)^2$ for which the $\GL_n(\C)$-additive diameter of $\sl_n(\C)$ with respect to $U$ is equal to $3$.

\begin{figure}[t]
    \begin{tikzpicture}
      \def\n{8}  
      
      \matrix[matrix of math nodes, nodes in empty cells] (m)
      {
        \cdot & \cdot & \cdot & \cdot & \cdot & \cdot & \cdot & \cdot \\
        \cdot & \cdot & \cdot & \cdot & \cdot & \cdot & \cdot & \cdot \\
        \cdot & \cdot & \cdot & \cdot & \cdot & \cdot & \cdot & \cdot \\
        \cdot & \cdot & \cdot & \cdot & \cdot & \cdot & \cdot & \cdot \\
        \cdot & \cdot & \cdot & \cdot & \cdot & \cdot & \cdot & \cdot \\
        \cdot & \cdot & \cdot & \cdot & \cdot & \cdot & \cdot & \cdot \\
        \cdot & \cdot & \cdot & \cdot & \cdot & \cdot & \cdot & \cdot \\
        \cdot & \cdot & \cdot & \cdot & \cdot & \cdot & \cdot & \cdot \\
      };
    
      \foreach \r in {1,...,7} {
        \foreach \c in {2,...,8} {
          \node[fill=gray!20] at (m-\r-\c) {$\cdot$};
        }
      }
      \node[fill=gray!20] at (m-1-1) {$\cdot$};
      \node[fill=gray!20] at (m-8-8) {$\cdot$};
    
      \node[fill=OliveGreen!20] at (m-7-2) {$\cdot$};

      \node[pattern=north west lines, pattern color=Maroon] at (m-1-1) {$\cdot$};
      \node[pattern=north west lines, pattern color=Maroon] at (m-2-2) {$\cdot$};

      \node[pattern=north west lines, pattern color=Maroon] at (m-7-7) {$\cdot$};
      \node[pattern=north west lines, pattern color=Maroon] at (m-8-8) {$\cdot$};
    \end{tikzpicture}
    \caption{The space $\cspan \langle E_{11}-E_{nn} \rangle + B_{n-1,2}$.}
    \label{figure: counterexample diameter 3}
\end{figure}
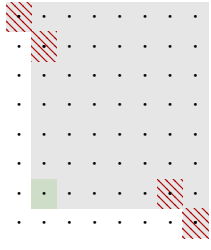

\begin{proposition}\label{prop:counterexample}
Let $U = \cspan \langle E_{11}-E_{nn} \rangle + B_{n-1,2} \leq \sl_n(\C)$ (see \Cref{figure: counterexample diameter 3}). Then 
\[
\dim U = (n-1)^2 \quad \text{and} \quad 
\diam_+^{\GL_n(\C)}(\sl_n(\C), U) = 3.
\]
\end{proposition}
\begin{proof}
Suppose there is some $g \in \GL_n(\C)$ with $U + g^{-1}U g = \sl_n(\C)$. 
Note that $U$ is stable under the action of the Borel subgroup of $\GL_n(\C)$. By the Bruhat decomposition \cite{borel2012linear}, we can write $g\in G$ as $bwb'$ where $b,b'$ belong to the Borel subgroup of $\GL_n(\C)$ and $w$ is an element of the corresponding Weyl group, which we may identify here with the symmetric group $S_n$, or  permutation matrices. 
Then $U+(bwb')^{-1}Ubwb'=b'^{-1}(U+w^{-1}Uw)b'=\sl_n(\C)$ (using that $U$ is invariant for $b,b'$) and hence $U+w^{-1}Uw=\sl_n(\C)$. 
Thus, we can view $g$ as a permutation $w \in S_n$.

Suppose first that $w(1) \neq n$. Then there exist matrices $A, B \in U$ such that
\[
    A + w^{-1} B w = E_{w^{-1}(n)1}.
\]
Hence we must have $w^{-1} E_{ij} w = E_{w^{-1}(n)1}$ for some $E_{ij} \in U$. This forces $j = w(1)$ and $w^{-1}(i) = w^{-1}(n)$, and so $i = n$, which is impossible since $E_{ij} \in U$.

Therefore, we must have $w(1) = n$. Now we can find, for any $2 \leq k \leq n$, matrices $A_k, B_k \in U$ such that
\[
    A_k + w^{-1}B_k\, w = E_{k1}.
\]
Repeating the argument from above, we must have $w^{-1} E_{i_k j_k} w = E_{k1}$ for some $E_{i_k j_k} \in U$, which forces $j_k = w(1) = n$ and $w^{-1}(i_k) = k$. Consequently, the set of indices $\{i_2, \dots, i_{n-1}\}$ must be equal to $\{ w(2), \dots, w(n-1)\}$. Note that as $E_{i_k j_k} \in U$, we must have $i_k \leq n-1$, and since $w(1) = n$, it follows that $\{i_2, \dots, i_{n-1}\} = \{ 2,\dots,n-1\}$. Therefore $w(n) = 1$. This means that the permutation $w$ swaps the first and the $n$-th positions, and so $w^{-1} U w$ is contained in the hyperplane $H = \{ X \in \mathfrak{sl}_n(\mathbb{C}) \mid X_{11}-X_{nn}=0 \}$. As $U \leq H$, we must have $U + w^{-1} U w \leq H$, a contradiction.

Having shown no two conjugates suffice to cover $\sl_n(\C)$, we now exhibit three conjugates that do. Flip $U$ through the center of the matrix to obtain $U + FUF = \cspan \langle E_{ij} \mid (i,j) \neq (1,1), (n,n) \rangle + \cspan \langle E_{11} - E_{nn} \rangle$. Let $w$ be the permutation matrix corresponding to the long cycle $(1 \ 2 \dots n)$. Then $w^{-1} U w$ contains $w^{-1}(E_{11} - E_{nn})w = E_{nn} - E_{n-1,n-1}$. Hence $U + FUF + w^{-1}Uw = \sl_n(\C)$.
\end{proof}

\section{Conjugation -- Largish Subspaces with Bounded Diameters}
\label{section: conjugation largish subspaces}

We continue studying the conjugation representation of $\GL_n(\C)$ on $\sl_n(\C)$. Here, we focus on diameters with respect to subspaces of dimension larger than $\epsilon n^2$. We prove diameter bounds that are optimal up to constant factors. 

We first show how some of the building blocks of upper right block closed subspaces quickly cover the whole $\sl_n(\C)$ with conjugation.

\begin{lemma} \label{diameter of the basic upper right block}
Let $m = \lfloor (n+1)/2 \rfloor$. 
Let (see \Cref{figure: basic_upper_right_blocks})
\[
B = \begin{cases}
    B_{m,m} & \text{if $n$ is odd}, \\
    B_{m,m} + B_{m+1,m+1} & \text{if $n$ is even}.
\end{cases}
\]
Then $\diam_+^{\SL_n(\C)}(\sl_n(\C), B) \leq 8$.
\end{lemma}

\begin{figure}[t]
  \begin{tikzpicture}
    \def\n{7}  
    
    \matrix[matrix of math nodes, nodes in empty cells] (m)
    {
      \cdot & \cdot & \cdot & \cdot & \cdot & \cdot & \cdot \\
      \cdot & \cdot & \cdot & \cdot & \cdot & \cdot & \cdot \\
      \cdot & \cdot & \cdot & \cdot & \cdot & \cdot & \cdot \\
      \cdot & \cdot & \cdot & \cdot & \cdot & \cdot & \cdot \\
      \cdot & \cdot & \cdot & \cdot & \cdot & \cdot & \cdot \\
      \cdot & \cdot & \cdot & \cdot & \cdot & \cdot & \cdot \\
      \cdot & \cdot & \cdot & \cdot & \cdot & \cdot & \cdot \\
    };
  
    \foreach \r in {1,...,3} {
      \foreach \c in {4,...,7} {
        \node[fill=gray!20] at (m-\r-\c) {$\cdot$};
      }
    }

    \foreach \r in {1,...,4} {
      \foreach \c in {5,...,7} {
        \node[fill=gray!20] at (m-\r-\c) {$\cdot$};
      }
    }
  \end{tikzpicture}
  \qquad
  \begin{tikzpicture}
      \def\n{8}  
      
      \matrix[matrix of math nodes, nodes in empty cells] (m)
      {
        \cdot & \cdot & \cdot & \cdot & \cdot & \cdot & \cdot & \cdot \\
        \cdot & \cdot & \cdot & \cdot & \cdot & \cdot & \cdot & \cdot \\
        \cdot & \cdot & \cdot & \cdot & \cdot & \cdot & \cdot & \cdot \\
        \cdot & \cdot & \cdot & \cdot & \cdot & \cdot & \cdot & \cdot \\
        \cdot & \cdot & \cdot & \cdot & \cdot & \cdot & \cdot & \cdot \\
        \cdot & \cdot & \cdot & \cdot & \cdot & \cdot & \cdot & \cdot \\
        \cdot & \cdot & \cdot & \cdot & \cdot & \cdot & \cdot & \cdot \\
        \cdot & \cdot & \cdot & \cdot & \cdot & \cdot & \cdot & \cdot \\
      };
    
      \foreach \r in {1,...,3} {
        \foreach \c in {4,...,8} {
          \node[fill=gray!20] at (m-\r-\c) {$\cdot$};
        }
      }

      \foreach \r in {1,...,4} {
        \foreach \c in {5,...,8} {
          \node[fill=gray!20] at (m-\r-\c) {$\cdot$};
        }
      }

      \foreach \r in {1,...,5} {
        \foreach \c in {6,...,8} {
          \node[fill=gray!20] at (m-\r-\c) {$\cdot$};
        }
      }
    \end{tikzpicture}
  \caption{The subspaces $B_{mm}$ ($n$ odd) and $B_{mm} + B_{m+1,m+1}$ ($n$ even). Note that both are contained in $\sl_n(\C)$, so the diagonal is zero.}
  \label{figure: basic_upper_right_blocks}
\end{figure}
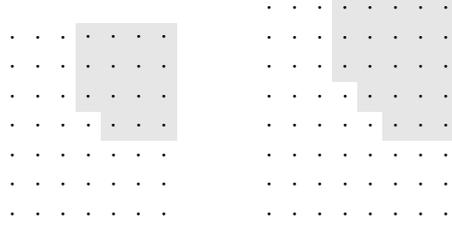

\begin{proof}
Flip $B$ through the center of the matrix. The space $B + FBF$ contains the space of matrices $V$ with block decomposition $(m-1, 1, n-m)$ whose diagonal blocks in this decomposition are all zero. 
It follows from \cite[Lemmas 2.6, 2.7]{brevsar2023waringISRAEL} that there is a matrix $g \in \GL_n(\C)$ with $\zerodiag_n(\C) \subseteq V + g^{-1}Vg$. Since $\diam_+^{\SL_n(\C)}(\sl_n(\C), \zerodiag_n(\C)) = 2$ by \Cref{diam of sl wrt zero diagonal}, we can conclude that $\diam_+^{\SL_n(\C)}(\sl_n(\C), B) \leq 8$.
\end{proof}
    
Upper right block subspaces of sufficiently large dimension contain the subspace $B$ from the previous lemma, so they also yield bounded diameters.

\begin{proposition}
Let $U \in \Gr(\sl_n(\C), d)$ be upper right block closed with $d > 3n^2/4 + n/2$. Then $\diam_+^{\GL_n(\C)}(\sl_n(\C), U) \leq 8$.
\end{proposition}
\begin{proof}
Let $\delta = d - 3n^2/4 - n/2 + 1/4 > 1/4$.
Let $k = \lfloor (n+1)/2 \rfloor$. 
Hence $\abs{k - (n+1)/2} \leq 1/2 < \sqrt{\delta}$.
Applying \Cref{upper right block from diagonal entry}, we obtain $B_{k+1,k-1} \leq U$, and so the result follows from \Cref{diameter of the basic upper right block}.
\end{proof}

This can be pushed further to cover all subspaces of dimension $\Omega(n^2)$ in a uniform way. Such subspaces contain a small $\sim \epsilon n \times \epsilon n$ upper right block, which we can then conjugate to tile the block $B$.

\begin{proposition}\label{epsilon n^2 proposition}
Let $0<\epsilon<1/3$ and $n\geq \max\{9,2/\epsilon\}$. If
$U\leq \sl_n(\C)$ has $\dim U>\epsilon n^2$, then
\[
  \diam_+^{\SL_n(\C)}
  \bigl(\sl_n(\C),U\bigr)
  \leq \frac{3456}{\epsilon}
  \left\lceil \log_2\frac{4}{\epsilon}\right\rceil .
\]
\end{proposition}

\begin{proof}
By \Cref{diameter bound reduction to Borel fixed subspaces} it is enough to consider $U$ stable under the standard
Borel subgroup. Recall that such a $U$ is upper-right block
closed, since if $u_{ij}\neq 0$, $i\neq j$, for some $u\in U$, then
\[
  B_{ij}:=\operatorname{span}\{E_{k\ell} \mid k\leq i,\ \ell\geq j\}
  \cap \sl_n(\C)\subseteq U.
\]
Let
\[
  S=\{(i,j) \mid i\neq j,\ u_{ij}\neq 0\text{ for some }u\in U\}.
\]
Since the traceless diagonal subspace has dimension $n-1$,
we have $\dim U\leq |S|+n-1$. As $n\geq 2/\epsilon$, it follows that $|S|> \epsilon n^2 / 2$.

Put $L=\lceil\log_2(4/\epsilon)\rceil$. At most
$n^2/2^L\leq \epsilon n^2/4$ elements of $S$ lie in rows
$i\leq n/2^L$. Thus one of the $L$ strips
\[
  \frac{n}{2^{r+1}}<i\leq \frac{n}{2^r}
  \qquad (0\leq r<L)
\]
contains at least $K=\epsilon n^2/(4L)$ elements of $S$. It contains at
most $n/2^r$ rows, so some row $p$ contains at least
$k\geq 2^rK/n$ such positions. If $q$ is the smallest corresponding column
and $c=n-q+1$, then $c\geq k$ and $p>n/2^{r+1}$. Therefore
\[
  pc\geq \frac{K}{2}\geq \frac{\epsilon n^2}{8L},
  \qquad B_{pq}\subseteq U.
\]

Now put $m=\lfloor(n+1)/2\rfloor$ and
$\alpha=\min\{p,m-1\}$, $\beta=\min\{c,m-1\}$. Since $n\geq 9$, we have
$m-1\geq n/3$, so
\begin{equation}
  \alpha\beta\geq \frac{pc}{9}
  \geq \frac{\epsilon n^2}{72L}.
  \label{eq:truncated-area}
\end{equation}
Moreover $\alpha+\beta\leq n$, and hence the coordinate rectangle
\[
  R_0=\operatorname{span}\{E_{ij} \mid 1\leq i\leq \alpha,\ 
  n-\beta+1\leq j\leq n\}
\]
is contained in $U$ and has disjoint row and column sets. Consequently,
permutation conjugation places a copy of $R_0$ on any two disjoint subsets of
$[n]$ having cardinalities $\alpha$ and $\beta$.

Let $(m-1,1,n-m)$  be the decomposition used in \Cref{diameter of the basic upper right block}, and
let $V$ be the space of matrices whose corresponding three diagonal blocks vanish. Its support is
contained in
\[
  [1,m-1]\times[m,n],\qquad [m,n]\times [1,m-1],\qquad
  \{m\}\times [n-m,n],\qquad [n-m,m]\times\{m\}.
\]
Partitioning these four rectangles into $\alpha\times\beta$ rectangles,
padding boundary pieces inside the corresponding disjoint parts, shows that
$V$ is contained in a sum of at most
\[
  2\left\lceil\frac{n}{\alpha}\right\rceil
   \left\lceil\frac{n}{\beta}\right\rceil
  +\left\lceil\frac{n}{\alpha}\right\rceil
  +\left\lceil\frac{n}{\beta}\right\rceil
  \leq \frac{12n^2}{\alpha\beta}
  \leq \frac{864L}{\epsilon}
\]
permutation conjugates of $U$, by \eqref{eq:truncated-area}.

Finally, the last step of \Cref{diameter of the basic upper right block} expresses $\sl_n(\C)$ as a sum of
four conjugates of $V$. Hence
\[
  \diam_+^{\SL_n(\C)}
  \bigl(\sl_n(\C),U\bigr)
  \leq 4\cdot\frac{864L}{\epsilon}
  =\frac{3456L}{\epsilon},
\]
as claimed.
\end{proof}

This bound is, up to absolute constants and the logarithmic factor, asymptotically optimal in $\epsilon$. If $\dim U \sim \epsilon n^2$, then even if all the conjugates of $U$ we use for covering $\sl_n(\C)$ are disjoint, we need, for dimension reasons, at least $\sim n^2/(\epsilon n^2) = 1/\epsilon$ of them.

\section{Lie-Additive Diameters}
\label{section: Lie-additive diameters}

\subsection{The three Lie-additive diameters}

Let $L$ be a Lie algebra with a Lie representation $\rho \colon L \to \gl(V)$. For $X \subseteq V$ and $a \in L$, define $\rho(a) \cdot X = \{ \rho(a) \cdot x \mid x \in X \}$. Analogous to the group-additive diameter, we can define the Lie-additive diameter of $V$ with respect to $X$. There are several natural but nonequivalent ways to define this concept. As we illustrate through examples below, some definitions are overly restrictive, others too broad, while one appears to strike the right balance.

In order to state the definitions, let $\ass(\rho(L))$ be the associative subalgebra of $\gl(V)$ generated by $\rho(L)$. Its elements are sums of scalar multiples of monomials $\rho(x_1) \cdots \rho(x_k)$, where $x_1, \dots, x_k \in L$ and $k \geq 0$. Let $\mon(\rho(L))$ be the set of all such monomials. The \emph{$L$-additive diameters} of $V$ with respect to $X$ are:
\begin{description}[leftmargin=1em, itemsep=0.5em,font=\normalfont\itshape,topsep=0.5em]
  \item[elementary] \hfill \\[0.5em]
  $\diam_+^{L}(V, X) = 
  \min \left\{ d+1 \mid
  V = X + r_1 \cdot X + \cdots + r_d \cdot X,
  \ r_i \in \rho(L)
  \right\}$,\footnote{We take $X$ as a separate summand to ensure that $X$ is contained in the sum.}
  \item[monomial] \hfill \\[0.5em]
  $\diam_+^{L,\mon}(V, X) = 
  \min \left\{ d \mid
  V = m_1 \cdot X + \cdots + m_d \cdot X,
  \ m_i \in \mon(\rho(L))
  \right\}$,
  \item[associative] \hfill \\[0.5em]
  $\diam_+^{L,\ass}(V, X) =
  \min \left\{ d \mid
  V = a_1 \cdot X + \cdots + a_d \cdot X,
  \ a_i \in \ass(\rho(L))
  \right\}$.
\end{description}
Clearly we always have
\[
\diam_+^{L,\ass}(V, X) \leq \diam_+^{L,\mon}(V, X) \leq \diam_+^{L}(V, X).
\]
Let us think through that when $V$ is an irreducible representation of $L$, it exhibits optimal associative Lie-additive diameters, so these diameters are, in a sense, useless.

\begin{lemma}
Let $V$ be an irreducible representation of $L$. Then $\ass(\rho(L)) = \gl(V)$, and so $V$ exhibits optimal associative Lie-additive diameters. 
\end{lemma}
\begin{proof}
The first part is immediate from Burnside's irreducibility theorem, and the second part follows since $\gl(V)$ acts transitively on subspaces of $V$ of the same dimension.
\end{proof}

We inspect the other two diameters for the case of irreducible representations of $\sl_2(\C)$ and the adjoint representation of $\sl_n(\C)$ on itself in the following sections. We give several examples when the monomial and elementary diameters differ. 

\subsection{Irreducible representations of $\sl_2(\C)$}

We inspect diameters of irreducible representations of $\sl_2(\C)$. We prove that the monomial diameter is always optimal, as in the group case. On the other hand, the elementary diameter can easily be infinite.

Smooth irreducible representations of $\sl_2(\C)$ can be realized on the same vector space $V = \C[X,Y]_k$ for $k \geq 1$ as with $\SL_2(\C)$, where $\sl_2(\C)$ acts by the derivative $D_I \rho_k$ of the standard irreducible representation $\rho_k$ of $\SL_2(\C)$. Write $e,h,f$ for the standard basis of $\sl_2(\C)$, and let $E_k, H_k, F_k \in \End(\C[X,Y]_k)$ be their images under $D_I \rho_k$. Then we have
\[
E_k \cdot e_i = (k-i) e_{i+1}, \quad
H_k \cdot e_i = (2i-k) e_i, \quad
F_k \cdot e_i = i e_{i-1}.
\]

\begin{theorem}
Every irreducible representation of $\sl_2(\C)$ exhibits optimal monomial Lie-additive diameters. 
\end{theorem}
\begin{proof}
By \Cref{diameter bound reduction to Borel fixed subspaces}, it suffices to prove the claim for upper closed subspaces $U = \langle e_j, e_{j+1}, \dots, e_k \rangle$ with $0 < j < k$. Let $d = \dim U = k - j + 1$. The operator $F_k$ maps $U$ into $\langle e_{j-1}, e_j, \dots, e_{k-1} \rangle$. Hence
\[
F_k^d \cdot U = \langle e_{j-d}, e_{j-d+1}, \dots, e_{k-d} \rangle.
\]
Thus $F_k^d$ moves the subspace $U$ downward in the basis, so after enough iterations it fills the whole space:
\[
\sum_{0 \leq i < (k+1)/d} \left( F_k^{di} \cdot U \right) 
= V. \qedhere
\]
\end{proof}

On the other hand, the elementary diameter can easily be infinite.

\begin{example}
Let $V = \C[X,Y]_k$ be the irreducible representation of $\sl_2(\C)$ of dimension $k+1 \geq 3$. Let $U = \langle e_j, e_{j+1}, \dots, e_k \rangle$ be an upper closed subspace with $j > 1$. Then $E_k, H_k$ preserve $U$, whereas $F_k \cdot U = \langle e_{j-1}, e_j, \dots, e_{k-1} \rangle$, and so $U + \sum_i D_I \rho_k (x_i) \cdot U \leq \langle e_{j-1}, e_j, \dots, e_k \rangle \neq V$ for all $x_i \in \sl_2(\C)$ and any number of them. Therefore the elementary Lie-additive diameter of $V$ with respect to $U$ is infinite.
\end{example}

These results suggest that out of the three Lie-additive diameters, $\diam_+^{L,\mon}$ is the most meaningful notion (finiteness of $\diam_+^L$ can fail even in simple cases, while $\diam_+^{L,\ass}$ is always optimal).

\subsection{Large subspaces with optimal diameters}

Let us show that all the critical subspaces from \Cref{critical upper right closed subspaces} have Lie-additive diameter $2$. For this, we reuse the matrix $F$ that induces flipping along the center of the matrix. For any $A \in \sl_n(\C)$, we have $[A,F] = F(FAF - A)$ and so $[A,F]_{ij} = (FAF)_{n+1-i,j} - A_{n+1-i,j} = A_{i,n+1-j} - A_{n+1-i,j}$. Hence $[E_{ij},F] = E_{i,n+1-j} - E_{n+1-i,j}$ (see \Cref{figure: adjoint with F}).

\begin{figure}[t]
  \begin{tikzpicture}
      \def\n{8}  
      
      \matrix[matrix of math nodes, nodes in empty cells] (m)
      {
        \cdot & \cdot & \cdot & \cdot & \cdot & \cdot & \cdot & \cdot \\
        \cdot & \cdot & \cdot & \cdot & \cdot & \cdot & \cdot & \cdot \\
        \cdot & \cdot & \cdot & \cdot & \cdot & \cdot & \cdot & \cdot \\
        \cdot & \cdot & \cdot & \cdot & \cdot & \cdot & \cdot & \cdot \\
        \cdot & \cdot & \cdot & \cdot & \cdot & \cdot & \cdot & \cdot \\
        \cdot & \cdot & \cdot & \cdot & \cdot & \cdot & \cdot & \cdot \\
        \cdot & \cdot & \cdot & \cdot & \cdot & \cdot & \cdot & \cdot \\
        \cdot & \cdot & \cdot & \cdot & \cdot & \cdot & \cdot & \cdot \\
      };
    
      \node[fill=OliveGreen!20] at (m-3-7) {$\cdot$};
      \node[fill=gray!20] at (m-3-2) {$\cdot$};
      \node[fill=gray!20] at (m-6-7) {$\cdot$};

    \draw[black, thick] 
      ($(m-4-1.south west)!0.5!(m-5-1.north west)$) -- 
      ($(m-4-8.south east)!0.5!(m-5-8.north east)$);

    \draw[black, thick] 
      ($(m-1-4.north east)!0.5!(m-1-5.north west)$) -- 
      ($(m-8-4.south east)!0.5!(m-8-5.south west)$);
    \end{tikzpicture}
  \caption{Taking adjoint with $F$ reflects $E_{ij}$ across the vertical and horizontal midlines.}
  \label{figure: adjoint with F}
\end{figure}
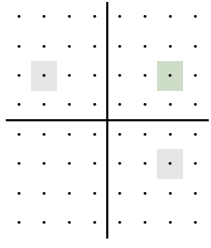

\begin{proposition}
Let $U \in \Gr(\sl_n(C), d)$ be upper right block closed with $d = (n-1)^2 + 1$ and $n\neq 4$.\footnote{In the case of $n=4$ there is an example with the diameter equal to $3$.} Then
\[
\diam^{\sl_n(\C)}_+(\sl_n(\C), U) = 2.
\]
\end{proposition}
\begin{proof}
Let us only prove the same conclusion for the case when $U = \borel$. All upper right block closed subspaces of dimension $(n-1)^2 + 1$ are described in \Cref{critical upper right closed subspaces} and can be handled in the same way. For any $i < j$ with $i + j > n + 1$, we have $\ad_F E_{ij} = E_{i,n+1-j} - E_{n+1-i,j}$, and so $E_{i,n+1-j} \in \borel + \ad_F \borel$. Similarly we obtain, for $i < j$ with $i + j < n + 1$, that $E_{n+1-i,j} \in \borel + \ad_F \borel$. Finally, for any $i > \lfloor (n+1)/2 \rfloor$, we have $\ad_F (E_{ii} - E_{n+1-i,n+1-i}) = 2 E_{i,n+1-i} - 2 E_{n+1-i,i}$, and so $E_{n+1-i,i} \in \borel + \ad_F(\borel)$. 
\end{proof}

By \Cref{diameter bound reduction to Borel fixed subspaces}, we thus obtain a precise analogue of \Cref{diameter 2 at large dimension} for all Lie-additive diameters.

\begin{theorem}
  Let $U \in \Gr(\sl_n(\C), d)$ with $(n-1)^2 < d < n^2 - 1$.
  Then 
  \[
  \diam_+^{\sl_n(\C)}(\sl_n(\C), U) = 2.
  \]
\end{theorem}

\subsection{Example: distinct diameters for $\SL_3(\C)$ and $\sl_3(\C)$}
  Let $G = \SL_3(C)$ act on $V = \sl_3(\C)$ by conjugation, and let $L = \sl_3(\C)$ act on itself by the adjoint representation. We give an example of a subspace $U \leq V$ with
  \[
  \diam_+^G(V,U) = \diam_+^L(V,U) = 3, \quad
  \diam_+^{L,\mon}(V,U) = \diam_+^{L,\ass}(V,U) = 2. 
  \]
  For this, we reuse the example from \Cref{prop:counterexample}: 
  \[
  U 
  = \cspan \langle E_{11}-E_{33}, E_{12}, E_{13},E_{23}\rangle.
  \] 
  By \Cref{prop:counterexample}, we have $\diam_+^G(V,U) = 3$. We also know that $\diam_+^{L,\ass}(V,U) = 2$, as it is always optimal. One can check (by hand or by computer) that 
  \[
  U + \ad_{E_{22}-E_{33}+E_{21}} \ad_{E_{21}+E_{31}+E_{32}} U = V,
  \]
  hence $\diam_+^{L,\mon}(V,U) = 2$. It remains to argue that $\diam_+^L(V,U) = 3$. We first claim that $U+ \ad_r U \neq V$ for any $r\in L$. To this end, note that $\ad_r{E_{13}} \in U$ for all $r\in L$ and so $\dim(U+ \ad_r U)\leq 7<8 = \dim V$. Conversely, we have 
  \[
  U + \ad_{E_{21}} U + \ad_{E_{31}} U = V,
  \]
  and so $\diam_+^L(V,U) = 3$.

\subsection{Lie versus group-additive diameters}

Suppose $G$ is a complex algebraic group with a representation $\rho \colon G \to \GL(V)$. This induces a corresponding Lie representation $D_I \rho \colon \g \to \gl(V)$. Given a subspace $V$, we do not know to what extent the $G$-additive diameter and the $\g$-additive diameter are related in general. 

In light of the examples and results discussed in the previous subsections (irreducible representations of $\SL_2(\C)$/$\sl_2(\C)$ and conjugation/adjoint representations on $\sl_n(\C)$ with respect to large subspaces), it seems plausible that we have the following situation.

\begin{question}
  Do irreducible representations always exhibit optimal monomial Lie-additive diameters?
\end{question}

\begin{question}
    Let $G$ be a complex connected algebraic group with an irreducible representation on $V$ and let $U \leq V$. Is it true that 
    \[
    \diam_+^{\g,\mon}(V, U) 
    \leq \diam_+^G(V, U) 
    \leq \diam_+^{\g}(V, U) \ ?
    \]
\end{question}

We can only give a very modest result in this direction for the case of $\SL_n(\C)$ acting by conjugation on $\sl_n(\C)$.

\begin{proposition}
    Let $\SL_n(\C)$ act on its Lie algebra $\sl_n(\C)$ by conjugation, and let $\sl_n(\C)$ act on itself by the adjoint representation.
    For any subspace $U$ of $\sl_n(\C)$,
    \[
        \diam_+^{\SL_n(\C)}(\sl_n(\C), U) \leq 8 \cdot \diam_+^{\sl_n(\C)}(\sl_n(\C), U) - 7.
    \]
\end{proposition}
\begin{proof}
\hspace{-0.75em} \footnote{We thank Matej Brešar for showing us this argument.}
Let $d + 1 = \diam_+^{\sl_n(\C)}(\sl_n(\C), U)$. Then there exist $r_1, \dots, r_d \in \sl_n(\C)$ such that $\sl_n(\C) = U + [r_1, U] + \cdots + [r_d, U]$. Since every matrix in $\sl_n(\C)$ is a sum of $4$ matrices in $\sl_n(\C)$ with square zero \cite[Theorem 3.6]{wang1991sums}, we can write $r_i = \sum_{1 \leq j \leq 4} x_{ij}$ with $x_{ij}^2 = 0$ for all $i,j$. Now, for any $x \in \sl_n(\C)$ with $x^2 = 0$, we have $(I - x)^{-1} = I + x$, and so for any $u \in U$, we can compute
\[
    (I + x) \cdot u = (I + x)^{-1} u (I + x) = u + [u,x] - xux
\]
and thus
\[
    2 [u,x] = (I + x)^{-1} u (I + x) - (I - x)^{-1} u (I - x) \in (I + x) \cdot U + (I - x) \cdot U.
\]
Therefore
\[
    [U, r_i] \leq \textstyle \sum_{1 \leq j \leq 4} \left( (I + x_{ij}) \cdot U + (I - x_{ij}) \cdot U \right).
\]
It now follows that
\[
    \sl_n(\C) = U + \textstyle \sum_{1 \leq i \leq d} \sum_{1 \leq j \leq 4} \left( (I + x_{ij}) \cdot U + (I - x_{ij}) \cdot U \right),
\]
so $\diam_+^{\SL_n(\C)}(\sl_n(\C), U) \leq 1 + 8d$.
\end{proof}

\section{Additive Diameters in Equivariant Morphisms}
\label{section: equivariant morphisms}

\subsection{Equivariant algebraic morphisms}

Let $G$ be a complex linear algebraic group. In this section, we show how to apply the diameter bounds developed in earlier sections to the images of $G$-equivariant algebraic morphisms $f \colon W \to V$, where $W$ and $V$ are representations of $G$. After choosing a basis for $V$, one can interpret $f$ as a tuple of polynomial functions on $W$. Thus, $f$ can be viewed as an element of $k[W]\otimes V$, and $G$-equivariance implies that $f$ actually lies in the invariant subring $(k[W]\otimes V)^G$. Alternatively, we can regard $f$ as a linear map in the space of $G$-equivariant maps
\[
  \bigoplus_{k \geq 0} \hom_G \left( \Sym^k (W), V \right),
\]
where $\Sym^k(W)$ denotes the $k$th symmetric power of $W$. In this correspondence, the summand $\hom_G (\Sym^k(W), V)$ represents $G$-equivariant homogeneous polynomial maps of degree $k$ from $W$ to $V$.\footnote{For instance, if $W = \C^2$ and $V = \C$, then the polynomial map $f(x,y) = x^2 + 2xy$ corresponds to $(e_1 \otimes e_1)^\ast + 2(e_1 \otimes e_2)^\ast$ in $\hom(\Sym^2(W),V)$.}
Therefore many such maps $f$ can be obtained by identifying subrepresentations of $\Sym^k(W)$ that are isomorphic to $V$.

\begin{example}
Let $G = \SL_2(\C)$ and let $W = \C[X,Y]_1 = \C^2$, the standard representation of $G$. Then $\Sym^k(W)$ is precisely the irreducible representation $\rho_k$ on $\C[X,Y]_k$. Therefore we have, up to a scalar, a unique algebraic morphism $f$ of degree $k$ from $\C^2$ to $\C[X,Y]_k$. For example, the case $k = 3$ gives the parameterization of the twisted cubic $f  \colon (s,t)\mapsto (s^3,s^2t,st^2,t^3)$. 
\end{example}

\begin{example}
Suppose that $\SL_n(\C)$ acts on $V = \sl_n(\C)$ by conjugation, and take $W = V^m = \sl_n(\C)^{\oplus m}$. Let $f \colon W \to V$ be an $\SL_n(\C)$-equivariant morphism. Then $f$ is a trace polynomial by \cite[Theorem 2.1]{Procesi}. For example, $f(X_1,\dots,X_m) = X_1^2X_2 - \tr(X_m)X_{3} + \tr(X_{m-1}^3)$.
\end{example}

\subsection{Images of derivatives of equivariant maps}

Let $G$ be a complex linear algebraic group with a representation $\rho \colon G \to \GL(V)$. For any $g \in G$, we have the derivative $D_g \rho \colon T_g G \to \gl(V)$ of $\rho$ at $g$.\footnote{Let $L_g \colon \GL(V) \to \GL(V)$ be left multiplication by $g$. Then $D_I L_g \colon \gl(V) = T_I \GL(V) \to T_g \GL(V)$ is an isomorphism.} 
In particular, for $g = I$, we obtain a representation of the Lie algebra $T_I G = \g$ of $G$. Say $\lambda \colon G \to \GL(W)$ is another representation of $G$. Let $f \colon W \to V$ be a differentiable map that is equivariant with respect to $\lambda$ and $\rho$, i.e., $f(\lambda(g) \cdot w) = \rho(g) \cdot f(w)$ for all $w \in W$ and $g \in G$. The following lemmas deal with understanding the image $\image (D_w f)$ in relation to derivatives of $\rho$.

\begin{lemma}
For any $v \in V$, let $\rho_v \colon G \to V$, $\rho_v(g) = \rho(g) \cdot v$. Then for any $g \in G$,
  \[
    D_g \rho_v \colon T_g G \to V, \quad
    x \mapsto (D_g \rho \cdot x) \cdot v.
  \]
\end{lemma}
\begin{proof}
Writing $\eval_v \colon \End(V) \to V$ for the evaluation map at $v$, we have
  \[
(D_g \rho_v) \cdot x =
D_g (\eval_v \circ \rho) \cdot x = 
\eval_v \cdot D_g \rho \cdot x = 
(D_g \rho \cdot x) \cdot v. \qedhere
\]
\end{proof}

\begin{lemma}
For any $w \in W$ and $g \in G$, we have
  \[
    \image (D_{\lambda(g) \cdot w} f) = \rho(g) \cdot \image (D_w f ).
  \]
\end{lemma}
\begin{proof}
Since $f$ is $G$-equivariant, we have $f(\lambda(g) \cdot w) = \rho(g) \cdot f(w)$. Differentiating this equality with respect to $w$, we obtain
  \[
      D_{\lambda(g) \cdot w} f \cdot \lambda(g) = \rho(g) \cdot D_w f.
  \]
It follows that $\image(D_{\lambda(g) \cdot w} f) = \rho(g) \cdot \image(D_w f)$, as claimed.
\end{proof}

\begin{lemma} \label{lemma on images of derivatives}
For any $w \in W$ and $g \in G$, we have
\[
  \image (D_g \rho_{f(w)}) \leq \rho(g) \cdot \image (D_w f).
\]
\end{lemma}
\begin{proof}
Letting $\gamma \colon (-\epsilon,\epsilon)\to G$ be any smooth curve with $\gamma(0)=g$ and $D_0 \gamma = x \in T_g G$, we have
\[
(\rho_{f(w)}\circ \gamma) (t)=\rho(\gamma(t)) \cdot f(w) = 
f(\lambda(\gamma(t)) \cdot w)
\]
by equivariance. Thus,
\[
D_g \rho_{f(w)} \cdot x =
D_g \rho_{f(w)} \cdot D_0 \gamma =
D_0 (t \mapsto f(\lambda(\gamma(t)) \cdot w)) \in \image (D_{\lambda(g) \cdot w} f).
\]
Since $x\in T_g G$ was arbitrary, we conclude that 
$\image (D_g \rho_{f(w)}) \leq \rho(g) \cdot \image (D_w f)$ by using the previous lemma.
\end{proof}

Let us gather the results above in a single proposition. 

\begin{proposition} \label{images of derivative gathered results}
Let $G$ be a complex linear algebraic group. Let $f \colon W \to V$ be a $G$-equivariant algebraic morphism, and let $\rho$ be the representation of $G$ on $V$. For any $w \in W$ and $g \in G$, we have:
\[
  \image (D_w f) \geq \rho(g^{-1}) \cdot \image (D_g \rho) \cdot f(w)
  \quad \text{and} \quad
  \dim \image (D_w f) \geq \dim (\rho(G) \cdot f(w)).
\]
\end{proposition}
\begin{proof}
It follows from the first lemma that $\image (D_g \rho_{f(w)}) = \image D_g \rho \cdot f(w)$, and then the first containment is immediate from the last lemma. As for the second inequality, we have $\image \rho_{f(w)} = \rho(G) \cdot f(w)$, and hence $\dim(\overline{\image \rho_{f(w)}}) = \dim(\rho(G) \cdot f(w))$. It now follows from \cite[Chapter III, Proposition 10.6]{hartshorne2013algebraic} that $\rank(D_g \rho_{f(w)}) \geq \dim(\rho(G) \cdot f(w))$ for all $g$ belonging to a nonempty Zariski open subset of $G$. By the previous lemma, we also have $\rank(D_g \rho_{f(w)}) \leq \dim \image (D_w f)$. This completes the proof.
\end{proof}

\begin{example}
Let $G = \SL_n(\C)$ act on its Lie algebra $\sl_n(\C)$ by conjugation. The derivative of this representation at the identity is the adjoint representation of $\sl_n(\C)$ on itself. Now let $f \colon W \to \sl_n(\C)$ be any equivariant map. Then for any $w \in W$, the proposition gives $\image (D_w f) \geq [\sl_n(C), f(w)] = \image (\ad_{f(w)})$. If there is an element in $\image f$ that is regular semisimple, then by $G$-equivariance there is also an element $d$ that is diagonal with distinct eigenvalues, and so $\image (\ad_d) = \M_n^0(\C)$. In this case, we thus obtain, for some $w \in W$, that $\image (D_w f) \geq \M_n^0(\C)$.
\end{example}

\subsection{Additive diameters in equivariant morphisms}

For any function $f \colon W \to V$ and positive integer $k$, we define the $k$-fold sum
\[
    f^{[k]} \colon W^k \to V, \quad
    (w_1, \dots, w_k) \mapsto f(w_1) + \cdots + f(w_k).
\]
Our aim is to show that the smallest $k$ for which $f^{[k]}$ is surjective is controlled by group-additive diameters of images of derivatives.\footnote{The main idea behind this connection in the context of groups goes back to Andrei Jaikin's paper \cite[Lemma 2.1]{jaikin2008verbal}. On the other hand, in the context of algebraic varieties, the connection between sums of points on a projective variety (secant variety) and tangent spaces is known as Terracini's lemma \cite[Proposition 10.10]{eisenbud20163264} (see also \Cref{remark on Terracini lemma}).}

\begin{theorem} \label{diam_+ <= 2 diam_+^G}
Let $G$ be a complex linear algebraic group. Let $f \colon W \to V$ be a $G$-equivariant algebraic morphism. Then, for any $w \in W$,
\[
  \diam_+(V, \image f) \leq 2 \cdot \diam^G_+(V, \image (D_w f)).
\]
\end{theorem}
\begin{proof}
Let $\lambda, \rho$ be the representations of $G$ on $W, V$. Let $g_1, \dots, g_k \in G$, and set $p = (\lambda(g_1) \cdot w, \dots, \lambda(g_k) \cdot w) \in W^k$. We have
  \[
      D_{p} f^{[k]} =
      D_{\lambda(g_1) \cdot w} f \oplus \dots \oplus D_{\lambda(g_k) \cdot w} f,
  \]
  and so its image $\image(D_{p} f^{[k]})$ is equal to
\[
    \image(D_{\lambda(g_1) \cdot w} f) + \dots + \image(D_{\lambda(g_k) \cdot w} f)
    = \rho(g_1) \cdot \image (D_w f) + \cdots + \rho(g_k) \cdot \image (D_w f)
\]
by \Cref{lemma on images of derivatives}.
Taking $k = \diam^G_+(W, \image (D_w f))$, we can thus find $g_1, \dots, g_k \in G$ so that $\image(D_{p} f^{[k]}) = W$. 
The map $f^{[k]}$ is then a submersion in some neighbourhood of the point $p$. By the implicit function theorem, it follows that $\image(f^{[k]})$ contains an open subset of $f^{[k]}(p)$. By Chevalley's theorem (see \cite[Chapter II, Exercise 3.19]{hartshorne2013algebraic}), the image of $f^{[k]}$ is constructible, therefore it contains a nonempty Zariski open dense subset of $V$. It now follows from Borel's trick (see \cite[Chapter I, 1.3]{borel2012linear}) that $(f^{[k]})^{[2]} = f^{[2k]}$ is surjective.
\end{proof}

The theorem is of course only useful if the group-additive diameter in question is finite. In other words, the image $\image (D_w f)$ should generate $V$ as a $G$-module. This is clearly not always the case, and might depend on the point $w$. The ideal situation is when $\image (D_w f)$ is very large for some $w$.

\begin{example}
Let $\GL_n(\C)$ act on $\M_n(\C)$ by conjugation. Let $f \colon \M_n(\C) \to \M_n(\C)$ be defined as $f(X) = X^2$. Then $D_0 f = 0$. On the other hand, $D_I f = 2I$, so $\image (D_I f) = \M_n(\C)$. By the previous theorem, $f^{[2]}$ is surjective. Every matrix is thus a sum of two squares (but not every matrix is itself a square).
\end{example}

\begin{example}
Let $f \colon \M_n(\C)^2 \to \M_n(\C)$ be defined as $f(X,Y) = I + [X,Y]$. Then $\image (f)$ is contained in the set of matrices of trace $n$, so there is no $k$ making $f^{[k]}$ surjective. Correspondingly, the derivative $D_{(A,B)} f$ maps into $\sl_n(\C)$, so the $\GL_n(\C)$-additive diameter of $\image (D_{(A,B)} f)$ is infinite.
\end{example}

The last example demonstrates that the image of $f$ may be contained in a translate of a subrepresentation. We will now argue that this is the only obstacle to the finiteness of the associated group-additive diameter. Therefore, if there is any hope of proving that $f^{[k]}$ is surjective, then additive diameters always provide an approach.

\begin{proposition} \label{im f not in subrep implies finite diameter wrt derivative}
Let $G$ be a complex linear algebraic group, and let $f \colon W \to V$ be a $G$-equivariant algebraic morphism. Suppose that $V$ has only finitely many subrepresentations. If $\cspan\langle \image (f - f(0)) \rangle = V$,
then for some $w \in W$,
\[
\diam_+^G(V, \image (D_w f)) < \infty.
\]
\end{proposition}
\begin{proof}
Let $Z_1, \dots, Z_\ell$ be all the proper subrepresentations of $V$. It suffices to prove that for any $1 \le i \le \ell$, it cannot happen that $\image(D_w f) \leq Z_i$ for all $w\in W$, since the subrepresentation generated by $\image (D_w f)$ must then be equal to $V$ and so the diameter is finite. For the sake of contradiction, assume that this does occur for some $i$. Let $\lambda_i \colon V \to \C$ be a linear functional with $\ker \lambda_i \geq Z_i$. Then we have, for all $w \in W$, that $\lambda_i \cdot D_w f = 0$. But then by the chain rule $D_w (\lambda_i \circ f) = 0$,
so the polynomial $\lambda_i \circ f$ is constant. This implies $\image(f) \subseteq f(0) + \ker (\lambda_i)$, contradicting our assumption that $\cspan\langle \image (f - f(0)) \rangle = V$. 
\end{proof}

\begin{remark} \label{remark on Terracini lemma}
  This can be compared with Terracini's lemma \cite[Proposition 10.10]{eisenbud20163264} from classical algebraic geometry in the case when $\image(f)$ is a homogeneous variety. If a projective variety $X \subset \PP(V)$ has the property that its $d$-th secant variety is the whole $\PP(V)$, then the sum of the tangent spaces at $d$ general points of $X$ equals $\PP(V)$. Our result can be seen as an equivariant analogue: if a homogeneous variety $X = \image(f)$ has finite additive diameter in $V$, then the sum of \emph{some} number of tangent spaces at suitable \emph{$G$-translates of a single point} also spans $V$. 
\end{remark}

We can use \Cref{images of derivative gathered results} to ensure $\image (D_w f)$ contains a special subspace or at least that $\dim \image (D_w f)$ is large. After that, we can apply our diameter bounds from the previous sections. Here are some examples of how this can be done.

\begin{example} \label{example sl2 reps twisted cubic}
  Let $G = \SL_2(\C)$. Let $W = \C^2$ be the standard representation and let $V = \C[X,Y]_k$ be the irreducible representation of $G$ of degree $k + 1$. Let $f \colon W \to V$ be the corresponding $G$-equivariant morphism. Note that for any nonzero $w \in W$, the derivative $D_w f$ is of rank $2$, so $\image(D_w f)$ is $2$-dimensional
  and then $\diam_+^G(V,\image(D_wf))=\lceil (k+1)/2 \rceil$ by Theorem \ref{thm:optimal_SL_2}. Hence $\diam_+(V,\image f)\leq 2\lceil (k+1)/2 \rceil$ by the previous theorem. In particular, the case $k = 3$ gives that every element of $\C^4$ is a sum of $4$ elements from the twisted cubic. 
  \end{example}  

\begin{example}
Suppose that $\SL_n(\C)$ acts on  $V=\sl_n(\C)$ by conjugation. Let $f \colon W \to \sl_n(\C)$ be an $\SL_n(\C)$-equivariant morphism (for example a trace polynomial). Suppose that for some $w \in W$, we have
\[
\dim C_{\sl_n(\C)}(f(w)) < (1 - \epsilon) n^2 - 1
\]
for some $0 < \epsilon < 1/3$. By \Cref{images of derivative gathered results}, we know $\image (D_w f) \geq \image (D_I \rho) \cdot f(w) = \image (\ad_{f(w)})$. It follows from the assumption on the dimension of the centralizer that $\dim \image (\ad_{f(w)}) > \epsilon n^2$. It now follows from \Cref{epsilon n^2 proposition}, as long as $n \geq \max (9, 2/\epsilon)$ that
\[
  \diam_+^{\SL_n(\C)}(\sl_n(\C), \image (D_w f)) \leq \frac{3456}{\epsilon} \left\lceil \log_2 \frac{4}{\epsilon} \right\rceil,
\]
and so \Cref{diam_+ <= 2 diam_+^G} implies that 
\[
  \diam_+(\sl_n(\C), \image f) \leq \frac{6912}{\epsilon} \left\lceil \log_2 \frac{4}{\epsilon} \right\rceil.
\]
In the extreme case, we could have $w$ with the property that $f(w)$ is regular semisimple (and so without loss of generality diagonal due to equivariance). In this situation, we have $\image (D_w f) \geq \M_n^0(\C)$, and so, by \Cref{diam of sl wrt zero diagonal},
\[
  \diam_+(\sl_n(\C), \image f) \leq 2 \cdot \diam_+^{\SL_n(\C)}(\sl_n(\C), \M_n^0(\C)) \leq 4.
\]
\end{example}

\begin{example}
Let $\GL_n(\C)$ act on $\M_n(\C)$ by conjugation. Let $f$ be a nonconstant noncommutative polynomial for which $\cspan \langle \image (f - f(0)) \rangle = \M_n(\C)$.\footnote{This means that $\image(f)$ is not contained in an additive coset of the scalars or of $\sl_n(\C)$.} Assuming $n > (2\deg f)^2$, there is, by \cite[Corollary 2.10]{brevsar2024matrix}, an element $f(w)$ that is diagonal with distinct eigenvalues. As in the previous example, we obtain $\image (D_w f) \geq \M_n^0(\C)$. It follows from \Cref{im f not in subrep implies finite diameter wrt derivative} that there is also an element $x \in \image (D_w f)$ with $x \notin \sl_n(\C)$. Hence
\[
  \diam_+(\M_n(\C), \image f) \leq 2 \cdot \diam_+^{\GL_n(\C)}(\M_n(\C), \M_n^0(\C) + \cspan \langle x \rangle) \leq 4.
\]
This bound should be compared with the results in \cite{brevsar2023waringISRAEL,brevsar2023waring,brevsar2024matrix}. In \cite{brevsar2023waring}, it was shown that for a noncommutative polynomial $f \colon \M_n(\C)^d \to \M_n(\C)$ of degree $\deg f \leq n + 1$, there exist scalars $\lambda_1, \lambda_2, \lambda_3$ such that $\sl_n(\C) \subseteq \lambda_1 \image f + \lambda_2 \image f + \lambda_3 \image f$. In other words,
\[
\diam^{\C^*}_+(\sl_n(\C), \image f) \leq 3.
\]
Shortly thereafter, in \cite{brevsar2024matrix}, it was further established that this diameter can in fact be bounded by $2$. Moreover, the authors prove the stronger statement that $\sl_n(\C) \subseteq \image f - \image f$. We note that the methods in these papers differ considerably from those presented here and do not yield results on the finiteness of the diameter with respect to $\image f$ alone (that is, without incorporating scalars or taking differences). 
\end{example}

\appendix

\section{Examples} \label{section: examples}

\begin{example}
  Let $f \colon \C^2 \to \C^4$ be the polynomial map given by $f(x,y) = (x^3, x^2 y, x y^2, y^3)$. Its image $\image(f)$ is a subvariety of $\C^4$, whose projectivization in $\PP^3$ is the twisted cubic $\mathcal C$. Its secant variety (the Zariski closure of the union of all lines between two points in $\mathcal C$) is the whole $\PP^3$ \cite[Proposition 10.11]{eisenbud20163264}. Since $\image(f)$ is a homogeneous variety, every point on the line between two points on it is a sum of two elements in $\image(f)$. It follows that the Zariski closure of the sumset $\image(f) + \image(f)$ is $\C^4$, and so $\diam_+(\C^4, \image(f)) \leq 4$ by Borel's trick \cite[Chapter I, 1.3]{borel2012linear}. Alternatively, the same conclusion follows from \Cref{example sl2 reps twisted cubic}. The diameter is certainly not equal to $2$, since the point $(0,1,0,0) \in \C^4$ is not a sum of two elements in $\image(f)$. In order to deduce that the diameter is in fact equal to $3$, a bit more work is needed (see \cite[Example 3.10]{carlini2014four}).
\end{example}

\begin{example}
Let $G = \GL_2(\C)$ act by conjugation on $V = \M_2(\C)$, and let $U = \cspan \langle I, E_{12} \rangle$. Every conjugate of $U$ contains $I$, so the diameter is at least $3$ for dimension reasons. On the other hand, let $F$ be the permutation matrix that swaps the two standard basis vectors. Then $F U F^{-1} = \cspan \langle I, E_{21} \rangle$, and let $E = I + E_{21}$. We have $E U E^{-1} = \cspan \langle I, -E_{11} + E_{12} - E_{21} + E_{22} \rangle$. This shows that $U + F U F^{-1} + E U E^{-1} = \M_2(\C)$. Therefore the $\GL_2(\C)$-additive diameter of $\M_2(\C)$ with respect to $U$ is $3$, which is not optimal.
\end{example}
  
\begin{example}
Let $G = \GL_n(\C)$ act by conjugation on $\sl_n(\C)$. This is an irreducible representation. Let $\zerodiag_n(\C)$ be the set of matrices with zero diagonal. Let us show that
\[
    \diam^{\GL_n(\C)}_+(\sl_n(\C), \zerodiag_n(\C)) = 2.
\]
Let $\D_n(\C)$ be the set of diagonal matrices in $\M_n(\C)$.
It suffices to find an element $g \in \GL_n(\C)$ such that the map 
\[
\zerodiag_n(\C) \to \D_n(\C) \cap \sl_n(\C), \quad z \mapsto \diag(g^{-1} z g)
\]
is surjective. We take $g = \prod_{i = 1}^{n-1} (I + E_{i,i+1})$. Note that $g = [\delta_{i \leq j}]_{i,j=1}^n$, and we have $g^{-1} = I + [- \delta_{j = i + 1}]_{i, j = 1}^n$. The $(i,i)$-term of $g^{-1} z g$ is equal to $e_i^T g^{-1} z g e_i$, which computes to
\[
(g^{-T} e_i)^T z (g e_i) =
(e_i - e_{i+1}) z (e_1 + e_2 + \cdots + e_i) =
\sum_{1 \leq k \leq i} (z_{i,k} - z_{i+1,k}).
\]
Let $v_\ell \in \sl_n(\C)$ be the matrix whose $\ell$-th row is equal to $[\delta_{j < \ell}]_{j = 1}^n$, and all other entries are zero. Then $v_\ell \in \zerodiag_n(\C)$ and the $(i,i)$-term of $g^{-1} v_{\ell} g$ is $-(\ell-1)$ if $i = \ell - 1$, $\ell-1$ if $i = \ell$, and $0$ otherwise. Thus $\diag(g^{-1} v_\ell g) = (\ell-1) (- e_{\ell-1} + e_\ell)$. Since the $- e_{\ell-1} + e_\ell$ span $\D_n(\C) \cap \sl_n(\C)$ as $\ell$ varies, the claim follows.
\end{example}

\begin{example}
Let $\GL_n(\C)$ act by conjugation on $\sl_n(\C)$. Let $U \leq \sl_n(\C)$ be the subspace of matrices with zero last column and row. We then have
\[
  \dim U = (n-1)^2 - 1 \quad \text{and} \quad \diam_+^{\GL_n(\C)}(\sl_n(\C), U) = 3.
\]
The argument goes as follows. Let $g \in \GL_n(\C)$. We first claim that $g \notin U + gUg^{-1}$. Indeed, if $g = u + gvg^{-1}$ for $u,v \in U$, then $I = g^{-1}u + v g^{-1}$. Applying this to $e_n$, we obtain 
\[
  \langle e_n, e_n \rangle 
  = \langle (g^{-1}u + v g^{-1})e_n, e_n \rangle
  = \langle v g^{-1} e_n, e_n \rangle
  = \langle g^{-1} e_n, v^\star e_n \rangle = 0
\]
since $v \in U$. This is a contradiction, and so $\diam_+^{\GL_n(\C)}(\sl_n(\C), U) \geq 3$. 

Now let $g = P_{1n}$ be the permutation matrix that exchanges $e_1$ and $e_n$, and let $h = P_{2n}$ be the permutation matrix that exchanges $e_2$ and $e_n$. Then $U + gUg^{-1}$ consists precisely of traceless matrices whose $(1,n)$ and $(n,1)$ entries are zero, and similarly $U + hUh^{-1}$ consists of traceless matrices whose $(2,n)$ and $(n,2)$ entries are zero. The sum of these two spaces is $\sl_n(\C)$. This completes the proof.
\end{example}

\bibliography{refs}
\bibliographystyle{alpha}

\end{document}